\documentclass[12pt]{amsart}
\usepackage{graphicx}
\usepackage{amsmath} \usepackage{amsthm}
\usepackage{amssymb}
\usepackage{amsfonts}
\usepackage{mathtools} \usepackage{xcolor} \usepackage{bm}
\usepackage{hyperref}
\usepackage{algorithm}
\usepackage[noend]{algpseudocode} \hypersetup{ colorlinks, citecolor=black, filecolor=black, linkcolor=black, urlcolor=black }
\usepackage{float} \usepackage{fullpage} \usepackage{bbm}
\usepackage{ytableau}

\ytableausetup{boxsize=2em} \usepackage{tikz}
\usetikzlibrary{calc}
\pgfmathsetseed{12345}
\allowdisplaybreaks[1]
\theoremstyle{plain}
\newtheorem{theorem}{Theorem}[section]
\newtheorem{proposition}[theorem]{Proposition}

\newtheorem{corollary}[theorem]{Corollary}
\newtheorem{lemma}[theorem]{Lemma}

\theoremstyle{definition}
\newtheorem{definition}[theorem]{Definition}
\newtheorem{innercustomdef}{Definition}
\newenvironment{primeddef}[1]
  {\renewcommand\theinnercustomdef{#1}\innercustomdef}
  {\endinnercustomdef}
\newtheorem{example}[theorem]{Example}
\newtheorem{remark}[theorem]{Remark}

\DeclareMathOperator{\SYT}{SYT}

\DeclareMathOperator{\SSYT}{SSYT}

\DeclareMathOperator{\shape}{shape}
\DeclareMathOperator{\Des}{Des}
\DeclareMathOperator{\des}{des}
\DeclareMathOperator{\fDes}{fDes}
\DeclareMathOperator{\fdes}{fdes}
\DeclareMathOperator{\sDes}{sDes}

\DeclareMathOperator{\maj}{maj}

\DeclareMathOperator{\fmaj}{fmaj}

\DeclareMathOperator{\inv}{inv}

\DeclareMathOperator{\sinv}{sinv}
\DeclareMathOperator{\sInv}{sInv}
\DeclareMathOperator{\winv}{winv}
\DeclareMathOperator{\wInv}{wInv}

\DeclareMathOperator{\grFrob}{grFrob}
\DeclareMathOperator{\Frob}{Frob}
\DeclareMathOperator{\Hilb}{Hilb}
\DeclareMathOperator{\sign}{sign}

\DeclareMathOperator{\Stir}{Stir}

\DeclareMathOperator{\Ind}{Ind}

\DeclareMathOperator{\GL}{GL}
\DeclareMathOperator{\stair}{\mathrm{st}}

\DeclareMathOperator{\OP}{\mathcal{OP}}
\DeclareMathOperator{\SOP}{\widehat{\OP}}
\DeclareMathOperator{\OMP}{\mathcal{OMP}}
\DeclareMathOperator{\SOMP}{\widehat{\OMP}}
\DeclareMathOperator{\OmP}{\mathcal{O}m\mathcal{P}}
\DeclareMathOperator{\SOmP}{\widehat{\OmP}}

\DeclareMathOperator{\QSym}{\mathcal{Q}\mathit{sym}}
\DeclareMathOperator{\stdLR}{\overrightarrow{\mathrm{std}}}
\DeclareMathOperator{\stdRL}{\overleftarrow{\mathrm{std}}}
\DeclareMathOperator{\Sym}{Sym}

\newcommand{\llambda}{\bm \lambda}
\newcommand{\ggamma}{\bm \gamma}
\newcommand{\mmu}{\bm \mu}
\newcommand{\bbeta}{{\bm \beta}}
\newcommand{\symm}{\mathfrak{S}}
\newcommand{\AAA}{\mathcal{A}}

\newcommand{\ii}{\mathbf{i}}
\newcommand{\Z}{\mathbb{Z}}
\newcommand{\C}{\mathbb{C}}
\newcommand{\N}{\mathbb{N}}

\newcommand{\x}{\mathbf{x}}

\newcommand{\PP}{{\bf P}}
\newcommand{\QQ}{{\bf Q}}
\newcommand{\UU}{{\bf U}}
\newcommand{\eps}{\varepsilon}

\definecolor{cbBlue}{RGB}{0,114,178}
\definecolor{cbVermillion}{RGB}{213,94,0}
\definecolor{cbOrange}{RGB}{230,159,0}
\definecolor{cbSkyBlue}{RGB}{86,180,233}
\definecolor{cbGreen}{RGB}{0,158,115}
\definecolor{cbPurple}{RGB}{204,121,167}

\newcommand{\RandomFerrers}[3]{%
  \pgfmathtruncatemacro{\A}{2+int(rnd*3)}%
  \pgfmathtruncatemacro{\B}{1+int(rnd*\A)}%
  \pgfmathtruncatemacro{\C}{1+int(rnd*\B)}%

  \begin{scope}[shift={(#1,#2)}]
\draw (0,0) -- (-\A,0) -- (-\A,-3) -- (-\A+\C,-3) -- (-\A+\C,-2) -- (-\A+\B,-2) -- (-\A+\B,-1) -- (0,-1) -- cycle;

    \node[anchor=north west, inner sep=2pt]
      at (-\A,0) {$Q^{(#3)}$};
  \end{scope}

  \pgfmathparse{#1-\A}
  \xdef\nextx{\pgfmathresult} \pgfmathparse{#2-3}
  \xdef\nexty{\pgfmathresult} }

\newcommand{\FerrersChain}[1]{
  \begin{tikzpicture}[scale=.45,line join=round]
\def\currentx{0} \def\currenty{0}
    \pgfmathtruncatemacro{\lastindex}{#1-1}

    \foreach \i in {0,...,\lastindex}{
      \RandomFerrers{\currentx}{\currenty}{\i}

\ifnum\i=0 \node[anchor=south east, inner sep=3pt]
          at (\nextx,\nexty) {$0$}; \fi

      \xdef\currentx{\nextx}
      \xdef\currenty{\nexty} }
  \end{tikzpicture}
}

\DeclareRobustCommand{\qbinom}{\genfrac[]{0pt}{}}

\title{The bigraded Frobenius series of the wreath product superspace coinvariant rings}

\author[S. Bhattacharya]{Sutanay Bhattacharya}
\address{Department of Mathematics\\
         University of California, San Diego, CA, USA}
\email{subhattacharya@ucsd.edu}

\author[S. Corteel]{Sylvie Corteel}
\address{Department of Mathematics\\
         University of California, Berkeley, CA, USA}
\email{corteel@berkeley.edu}

\author[Y. Jiang]{Yuhan Jiang}
\address{Department of Mathematics\\
         University of California, Berkeley, CA, USA}
\email{jyh@math.berkeley.edu}

\author[J. Lentfer]{John Lentfer}
\address{Department of Mathematics\\
         University of California, San Diego, CA, USA}
\email{jlentfer@ucsd.edu}

\author[B. Rhoades]{Brendon Rhoades}
\address{Department of Mathematics\\
         University of California, San Diego, CA, USA}
\email{bprhoades@ucsd.edu}

\begin{document}

\begin{abstract}
	We determine the bigraded Frobenius series of the superspace coinvariant ring $R^{(1,1)}_{\mathbb{Z}_r\wr\mathfrak{S}_n}$ for the wreath products $\mathbb{Z}_r\wr\mathfrak{S}_n$, which is the quotient of the ring $\C[x_1,\ldots,x_n]\otimes\bigwedge[\theta_1,\ldots,\theta_n]$ of polynomial-valued differential forms by the ideal generated by diagonal $\mathbb{Z}_r\wr\mathfrak{S}_n$-invariants with vanishing constant term. The bigraded multiplicity of each irreducible character is a sum over standard Young $r$-tableaux of an expression determined by the flag major index and flag descent number.
	This refines the Frobenius series graded by fermionic degree found by Bhattacharya--Rhoades (2026). It specializes at $r=2$ to the hyperoctahedral group $\mathfrak{B}_n$ case, and it specializes at $r=1$ to recover the theorem of Murai--Rhoades--Wilson (2025) on the symmetric group $\mathfrak{S}_n$.
	
     We give new formulas for the bigraded multiplicities of hook characters of the superspace coinvariant ring in types $A$ and $B$, yielding upper bounds on trigraded multiplicities of the sign character of $R^{(1,2)}_{\mathfrak{S}_n}$ and the determinant character of $R^{(1,2)}_{\mathfrak{B}_n}$. 
     As a corollary, we deduce a conjecture of F. Bergeron (2024) on hook components of $R^{(1,1)}_{\mathfrak{S}_n}$. 
\end{abstract}

\maketitle

\section{Introduction}

Coinvariant rings are central objects in algebra and combinatorics, with connections to representation theory
\cite{Chevalley, MR59914, MR526968, Bergeron2020, BergeronOPAC, Lentfer-Supersymmetry}, algebraic geometry
\cite{BGG1973, Haiman2001, Haiman2002, Demazure1973, DeConciniProcesi}, differential forms
\cite{MR154929,SwansonWallach1,SwansonWallach2}, and link homology \cite{Khovanov, KhovanovRozansky1,
	KhovanovRozansky2, Hog17, Mellit, GorskyMellit, GonzalezLentferMularczyk}. The
classical coinvariant ring $R_{\mathfrak{S}_n}^{(1,0)}$ and the type $B$ coinvariant ring $R_{\mathfrak{B}_n}^{(1,0)}$,
which are isomorphic to the cohomology rings of the type $A$ and type $B$ flag varieties \cite{Borel,Leray}, are
special cases of coinvariant rings $R_{\Z_r \wr \mathfrak{S}_n}^{(1,0)}$ for wreath products, whose graded Frobenius series
was determined by Stembridge \cite{Stembridge}. (The definition of the coinvariant rings $R_G^{(k,j)}$ is given in
equation~\eqref{eq:coinvariant-ring-def}.)

Taking two sets of commutative variables instead gives the diagonal coinvariant ring
$R_{\mathfrak{S}_n}^{(2,0)}$ and its type $B$ variant $R_{\mathfrak{B}_n}^{(2,0)}$, which were introduced by Haiman
\cite{Haiman1994}. Using novel results on the geometry of the Hilbert scheme of points in $\C^2$ \cite{Haiman2001},
Haiman determined that the bigraded Frobenius series of $R_{\mathfrak{S}_n}^{(2,0)}$ is $\nabla e_n$, and
consequently, its dimension is $(n+1)^{n-1}$ \cite{Haiman2002}. Carlsson--Oblomkov later established a monomial basis
\cite{CarlssonOblomkov}. In type $B$, while the exact dimension of $R_{\mathfrak{B}_n}^{(2,0)}$ is unknown, Haiman
conjectured that it is at least $(2n+1)^n$ \cite{Haiman1994}. Gordon confirmed this bound using certain representations
of a rational Cherednik algebra \cite{Gordon}; the bound was further improved by Ajila--Griffeth \cite{AjilaGriffeth}.

For $\Z_r \wr \mathfrak{S}_n$ with $r \in \{1,2\}$, the reflection representation $V$ is isomorphic to its dual $V^*$, but
when $r \geq 3$, $V \not\cong V^*$. Thus it is possible to define different types of coinvariant rings depending on
whether $V$ or $V^*$ is used (see Section~\ref{subsec:coinvariants} for discussion). A diagonal coinvariant ring for wreath products,
arising as a quotient of $\C[V \oplus V^*]$, was studied by Griffeth \cite{Griffeth}, who bounded its dimension from below.

More than three decades after the seminal paper of Haiman \cite{Haiman1994}, the exact dimension, Hilbert series, and
bigraded Frobenius series are still unknown for $R_{\mathfrak{B}_n}^{(2,0)}$ and the wreath product cases, even conjecturally.

The diagonal fermionic coinvariant rings $R_{\mathfrak{S}_n}^{(0,2)}$ and $R_{\mathfrak{B}_n}^{(0,2)}$ are defined similarly, but with two sets of anticommuting variables.  Kim--Rhoades \cite{MR4381935} studied them in a
type-independent way as quotients of $\bigwedge(V \oplus V^*)$.
The bigraded multiplicity of each irreducible representation in $R_{\mathfrak{B}_n}^{(0,2)}$ was elucidated in \cite{JiangLentfer}. The
multiplicity of the sign character of $R_{\mathfrak{S}_n}^{(0,3)}$ was established in
\cite{lentfer2026signcharactertriagonalfermionic}, and that of the determinant character of
$R_{\mathfrak{B}_n}^{(0,3)}$ was established in \cite{JiangLentfer}.

By incorporating both commutative and anticommutative variables, one obtains the superspace coinvariant ring
$R_{\mathfrak{S}_n}^{(1,1)}$, first studied by N.~Bergeron, Colmenarejo, Li, Machacek, Sulzgruber, and Zabrocki. It was
the subject of conjectures by them, and by Sagan--Swanson \cite{MR4674564}. The multiplicity of the sign character of
$R_{\mathfrak{S}_n}^{(1,1)}$ \cite{SwansonWallach1}, its dimension and Hilbert series \cite{RhoadesWilson2023}, a
monomial basis \cite{ACKMR2024}, and its bigraded Frobenius series \cite{fields2025} have all been determined. In
particular, Murai--Rhoades--Wilson \cite{fields2025} showed that the bigraded Frobenius series of
$R_{\mathfrak{S}_n}^{(1,1)}$ may be written equivalently as
\begin{equation}\label{eq:type-A-product-form}
	\grFrob\big(R^{(1,1)}_{\mathfrak{S}_n};q;z\big) = \sum_{T \in \SYT_n} q^{\maj(T)} \prod_{j=1}^{\des(T)} \left(1+\frac{z}{q^{j}}\right)\, s_{\shape(T)}(\x)
	= \sum_{k=1}^{n} z^{n-k}\, \Delta'_{e_{k-1}} e_n \big|_{t=0}.
\end{equation}
Importantly, $\Delta'_{e_k}$ is an eigenoperator of the modified Macdonald basis, generalizing the $\nabla$ operator \cite{BergeronGarsiaHaimanTesler, HaglundRemmelWilson2018}.
The full trigraded symmetric function $\sum_{k=1}^{n} z^{n-k}\, \Delta'_{e_{k-1}} e_n$ is a sum of symmetric functions appearing in the Delta conjecture of Haglund, Remmel, and Wilson \cite{HaglundRemmelWilson2018}, whose `rise version' was proven first by D'Adderio--Mellit \cite{DAdderioMellit} and subsequently also by Blasiak, Haiman, Morse, Pun, and Seelinger \cite{BHMPS-Delta}.
Zabrocki \cite{Zabrocki2019} conjectured that the trigraded Frobenius series of $R_{\mathfrak{S}_n}^{(2,1)}$ is
\begin{equation}\label{eq:zabrocki-conjecture}
	\grFrob(R_{\mathfrak{S}_n}^{(2,1)};q,t;z) = \sum_{k=1}^{n} z^{n-k}\, \Delta'_{e_{k-1}} e_n.
\end{equation}
Since $R^{(1,1)}_{\mathfrak{S}_n}$ is the $t$-degree zero part of $R^{(2,1)}_{\mathfrak{S}_n}$, the theorem of Murai--Rhoades--Wilson is the $t=0$ case of equation~\eqref{eq:zabrocki-conjecture}.
The sign-isotypic component of Zabrocki's conjecture was established in \cite{GonzalezLentferMularczyk}, that is, the trigraded multiplicities of $s_{(1^n)}$ on both sides of equation~\eqref{eq:zabrocki-conjecture} agree.

The Hilbert series of the type $B$ ring $R_{\mathfrak{B}_n}^{(1,1)}$ was conjectured by Swanson--Wallach
\cite{SwansonWallach2}; a monomial basis was conjectured by Sagan--Swanson \cite{MR4674564}. The multiplicity of the
determinant character of $R_{\mathfrak{B}_n}^{(1,1)}$ was found by Swanson--Wallach \cite{SwansonWallach1}, while the
dimension and Hilbert series of $R_{\mathfrak{B}_n}^{(1,1)}$ were established by the first author in \cite{bhattacharya}. These results
were extended to wreath product groups by the first and last authors in \cite{BR}, 
resolving conjectures of Sagan--Swanson \cite{SSreflection}.
Furthermore, the ungraded Frobenius series of $R_{\Z_r \wr \mathfrak{S}_n}^{(1,1)}$ and its Frobenius series graded by
fermionic degree were established in \cite{BR}. The bases constructed in \cite{BR} determine the inner product of the
bigraded Frobenius series of $R_{\Z_r \wr \mathfrak{S}_n}^{(1,1)}$ with each of the symmetric functions $e_{\llambda}$.
However, no explicit formula for the bigraded multiplicities of the irreducible characters, i.e., the bigraded Frobenius series
in the Schur basis, was known or even conjectured for $R_{\Z_r \wr \mathfrak{S}_n}^{(1,1)}$ for $r \geq 2$. In this
paper, we establish such a result for all wreath product groups $\Z_r \wr \mathfrak{S}_n$. At $r=2$ it gives the bigraded
Frobenius series of $R_{\mathfrak{B}_n}^{(1,1)}$, and at $r=1$ it specializes to the first equality of equation~\eqref{eq:type-A-product-form}.

\begin{theorem}\label{thm:main-wreath}
	Fix integers $r,n \geq 1$. The bigraded Frobenius series of $R^{(1,1)}_{\Z_r\wr\mathfrak{S}_n}$ is given by
	\begin{equation}
		\grFrob(R^{(1,1)}_{\Z_r\wr\mathfrak{S}_n};q;z)=\sum_{\QQ\in \SYT^r_n}q^{\fmaj(\QQ)}\prod_{j=1}^{\fdes(\QQ)}\left(1+\frac{z}{q^{rj-r+1}}\right)s_{\shape(\QQ)}(\x),
	\end{equation}
	where $\SYT^r_n$ is the set of standard Young $r$-tableaux of size $n$, and $\fmaj(\QQ)$ and $\fdes(\QQ)$ are the flag major index and the number of flag descents of $\QQ$.
\end{theorem} 
Here the irreducible representations of $\Z_r\wr\mathfrak{S}_n$ are indexed by $r$-partitions as in Section~\ref{ss:rep_wr}. For $r \geq 3$ the formula depends on this indexing and on the action of $\Z_r \wr \mathfrak{S}_n$ on the variables, which are fixed in equation~\eqref{eq:action-on-variables} and Remark~\ref{rmk:conventions}.

We give some applications.
First, we extract the bigraded multiplicities of the hook
characters in types $B$ and $A$ (Propositions~\ref{prop:type-B-hooks-11} and~\ref{prop:type-A-hooks-11}).
Second, we use these formulas to bound the trigraded multiplicities of the determinant character of $R_{\mathfrak{B}_n}^{(1,2)}$ and the sign character of $R_{\mathfrak{S}_n}^{(1,2)}$ (Corollaries~\ref{cor:B12} and~\ref{cor:A12}); the ungraded bounds are, respectively, $3^n$ and $3^{n-1}$.
The type $A$ hook multiplicities also establish a conjecture of F.~Bergeron on the hook components of the multivariate diagonal coinvariants \cite[Conjecture~4]{BergeronOPAC} (Corollary~\ref{cor:bergeron-conjecture-4}).
Finally, we recover the Hilbert series of $R^{(1,1)}_{\Z_r\wr\mathfrak{S}_n}$, conjectured by Sagan--Swanson \cite{SSreflection} and proven by Bhattacharya--Rhoades \cite{BR}.

The organization of the paper is as follows. In  Section~\ref{sec:background}, we review relevant background: colored
permutations, $r$-tableaux, representation theory of $\Z_r\wr\mathfrak{S}_n$, colored RSK, coinvariant rings, colored quasisymmetric functions, and colored ordered set partitions. In Section~\ref{sec:main-results-wreath}, we prove our main result,
Theorem~\ref{thm:main-wreath}. In Section~\ref{sec:specialization-B}, we compute the hook multiplicities in types $B$ and $A$, then bound the determinant and sign multiplicities respectively of $R_{\mathfrak{B}_n}^{(1,2)}$ and $R_{\mathfrak{S}_n}^{(1,2)}$.
In Section~\ref{sec:combinatorial-consequences}, we recover the Hilbert series directly from our bigraded Frobenius series formula.

\section{Background}\label{sec:background}
Let $r$ be a positive integer, and let $\zeta := e^{2\pi\ii/r}$ be a primitive $r$th root of unity, where $\ii =
	\sqrt{-1}$. Let $\Z_r := \{1,\zeta,\dots,\zeta^{r-1}\}$ be the cyclic group of $r$th roots of unity. An \emph{$r$-colored positive integer} is a formal symbol $i^j$ with $i \in \Z_{>0}$ and $0 \leq j \leq r-1$; we will refer to
$i$ and $j$ as the \emph{base} and the \emph{color} of $i^j$ respectively. (One may also think of $i^j$ as the complex number $\zeta^j i$.)  Write $\N_r$ for the set of $r$-colored positive integers. We also use the symbol $0$, to which we always assign the color $0$. Consider the alphabet of
\emph{$r$-colored nonnegative integers}
$\{0\} \cup \N_r$
with the total order $<_r$ defined by comparing colors first, with higher color being smaller, and bases second:
\begin{equation}\label{eq:r-order}
	1^{r-1} <_r 2^{r-1} <_r\cdots <_r 1^{r-2} <_r 2^{r-2}<_r\cdots <_r 0 <_r 1^0 <_r 2^0 <_r \cdots.
\end{equation}

\subsection{Colored words and \texorpdfstring{$r$}{r}-tableaux}\label{subsec:tableaux}

An \emph{$r$-colored word} of length $n$ is a word $g = g_1 \cdots g_n$ in the alphabet $\N_r$.
We write $g_i = w_i^{\eps(i)}$, so that $w_i$ is the base and $\eps(i)$ is the color of $g_i$. Define $\eps(g):= \eps(1) + \cdots + \eps(n)$ for the sum of the colors of $g$.
An \emph{$r$-colored permutation} is an $r$-colored word of length $n$ in which each of $1,\ldots,n$ appears exactly once as a base.
Equivalently, it is a pair $g=(w,\eps)$ of a permutation $w \in \mathfrak{S}_n$ and a color function $\eps: [n] \to \{0,\ldots, r-1\} $.
In one-line notation, $g$ is written as $g_1\cdots g_n$, with $g_i = w(i)^{\eps(i)}$.
The $r$-colored permutations form the wreath product group $\Z_r \wr \mathfrak{S}_n$ of $\Z_r$ and the symmetric group $\mathfrak{S}_n$ (see Section~\ref{ss:rep_wr}). 
When $r=2$, $\Z_2 \wr \mathfrak{S}_n$ is the type $B$ Weyl group, the hyperoctahedral group $\mathfrak{B}_n$.

An index $0 \leq i \leq n-1$ is a \emph{flag descent} of an $r$-colored word $g$ if $g_i >_r g_{i+1}$, with
the convention that $g_0 = 0$. We denote the flag descent set of $g$ by $\fDes(g)$, and we denote the descent set of
$g$ by $\Des(g) = \fDes(g) \setminus \{0\}$. The number of flag descents is denoted by $\fdes(g) = |\fDes(g)|$. 
The \emph{major index} and the \emph{flag major index} of $g$ are defined by
\begin{equation}
	\maj(g) := \sum_{i \in \Des(g)} i, \qquad \text{and} \qquad \fmaj(g):= r \cdot \maj(g) + \eps(g).
\end{equation} 

For $r$-colored permutations, $\fmaj$ is the flag major index of Adin--Roichman \cite{MR1829737} (for $r=2$, see also \cite{MR3629266}).
Its extension to $r$-colored words is due to Haglund--Loehr--Remmel \cite{haglund2005statistics}.

\begin{example}\label{ex:fdes}
	Let $r=3$, and $g=3^12^02^04^23^21^0$ be a three-colored word.
	Then $\fDes(g)=\{0,3,4\}$ since $0>_r
		3^1$, $2^0>_r4^2$ and $4^2>_r3^2$ per our ordering. Consequently, $\Des(g)=\{3,4\}$, and $\fdes(g)=3$. The sum of
	colors in $g$ evaluates to $\eps(g)=1+0+0+2+2+0=5$, and so $\fmaj(g)=3\cdot(3+4)+5=26$.
\end{example}

Define an \emph{$r$-partition} of $n$ by $\llambda = (\lambda^{(0)},\ldots,\lambda^{(r-1)})$,
an $r$-tuple of partitions such that $|\lambda^{(0)}| + \cdots + |\lambda^{(r-1)}| = n$.
Denote this by $\llambda \vdash_r n$.

\begin{definition}\label{def:r-tableaux}
	A \emph{semistandard $r$-tableau} of shape $\llambda$ (then we write $\shape(\QQ) := \llambda$) is an $r$-tuple of semistandard Young tableaux $\QQ = (Q^{(0)},\ldots,Q^{(r-1)})$ such that $Q^{(i)}$ has shape $\lambda^{(i)}$ for $0 \leq i \leq r-1$.
    We will denote by $\SSYT^r(\llambda)$ the set of all semistandard $r$-tableaux of shape $\llambda$, and by $\SSYT^r_n$ the set of all semistandard $r$-tableaux of size $n$.

	A \emph{standard $r$-tableau} of shape $\llambda$ is a semistandard $r$-tableau of shape $\llambda$ where every
	element of $[n]$ appears exactly once. We will denote by $\SYT^r(\llambda)$ the set of all standard $r$-tableaux of
	shape $\llambda$, and by $\SYT^r_n$ the set of all standard $r$-tableaux of size $n$.
\end{definition}

We visualize an $r$-tableau by arranging $Q^{(0)},\ldots,Q^{(r-1)}$ along a diagonal, with $Q^{(0)}$ in the upper
right, and with $0$ between $Q^{(0)}$ and $Q^{(1)}$, as in Figure~\ref{fig:random-4-tableaux}.
\begin{figure}[H]
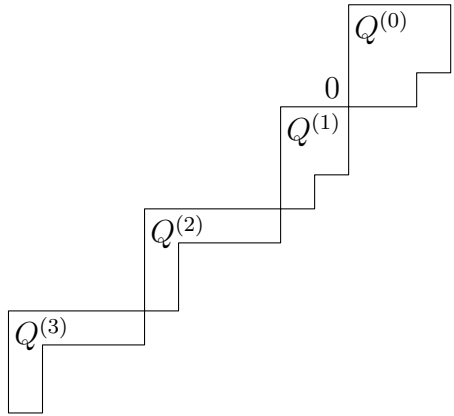

	\centering \FerrersChain{4}
	\caption{The skew shape used to visualize a $4$-tableau: $Q^{(0)}$ is in the upper right, each $Q^{(i+1)}$ lies below and to the left of $Q^{(i)}$, and a $0$ is placed between $Q^{(0)}$ and $Q^{(1)}$.}
	\label{fig:random-4-tableaux}
\end{figure}
The \emph{flag descent set} of a standard $r$-tableau $\QQ$ is the descent set of this skew tableau, with the convention that $0$ is a descent when $1$ lies below it. Explicitly:
\begin{definition}\label{def:fDes-sDes}
	An index $1 \leq i \leq n-1$ is a \emph{descent} of a standard $r$-tableau $\QQ$ if $i+1$ appears in a row below $i$ in the same tableau, or $i \in Q^{(j)}$ and $i+1 \in Q^{(k)}$ for $j < k$.
	We denote the \emph{descent set} of $\QQ$ by $\Des(\QQ)$.

	We say that $0$ is a \emph{flag descent} of $\QQ$ if $1 \notin Q^{(0)}$.
	We define the \emph{flag descent set} by
    \begin{equation}
        \fDes(\QQ) := \begin{cases}
            \Des(\QQ) \cup \{0\} & \text{ if $0$ is a flag descent of $\QQ$},\\
            \Des(\QQ) & \text{ otherwise}.
        \end{cases}
    \end{equation}
	The number of flag descents of $\QQ$ is denoted by $\fdes(\QQ) = |\fDes(\QQ)|$.

	The \emph{major index} and the \emph{flag major index} of $\QQ$ are defined by
	\begin{equation}\label{eq:fmajT}
		\maj(\QQ):= \sum_{i \in \Des(\QQ)} i, \qquad \text{and} \qquad \fmaj(\QQ):= r \maj(\QQ) + \sum_{j=0}^{r-1} j |Q^{(j)}|.
	\end{equation}

	Define the coloring function $\eps$ on $[n]$ such that $\eps(i) = j$ if $i \in Q^{(j)}$.
	The \emph{signed descent set} of $\QQ$ is the pair $\sDes(\QQ) = (\Des(\QQ), \eps)$.
\end{definition}

\subsection{Representation theory of wreath products}\label{ss:rep_wr}

The following is adapted from \cite[Chapter I, Appendix B]{Macdonald} and \cite{Stembridge}.
An $r$-colored permutation $g = (w,\eps)$ may be identified with the $n \times n$ monomial matrix whose $(w(i),i)$ entry
is $\zeta^{\eps(i)}$ for $i \in [n]$ and whose other entries are $0$. Then the set of all $r$-colored permutations may
be identified with the group $\Z_r \wr \mathfrak{S}_n$. Let $\sign(g):= \sign(w)$ be the sign of the underlying permutation
matrix, and let $\chi(g):= \zeta^{\eps(g)}$ be the product of the nonzero entries of $g$, so that $\det(g) = \sign(g)\chi(g)$.

Let $\{e_1,\ldots,e_n\}$ denote the standard basis of $V = \C^n$. The group $\Z_r \wr \mathfrak{S}_n$ acts on $V$ by $g \cdot
	e_i = \zeta^{\eps(i)} e_{w(i)}$, and consequently on polynomial functions on $V$ by $(g \cdot f)(v):= f(g^{-1}v)$. On
coordinates $x_1,\ldots,x_n$, this yields
\begin{equation}\label{eq:action-on-variables}
	g \cdot x_i = \zeta^{-\eps(i)} x_{w(i)}.
\end{equation}

Let $\x = (\x^{(0)},\ldots,\x^{(r-1)})$ be $r$ sets of infinitely many variables: $\x^{(i)} =
	(x^{(i)}_1,x^{(i)}_2,\ldots)$. 
    For an $r$-colored word $g = g_1\cdots g_n$ with $g_k = w_k^{\eps_k}$, and for a semistandard $r$-tableau $\QQ$, write
\begin{equation}\label{eq:word-and-tableau-monomials} \x^{g} := x^{(\eps_1)}_{w_1} x^{(\eps_2)}_{w_2} \cdots x^{(\eps_n)}_{w_n}, \qquad \x^{\QQ} := \prod_{i=0}^{r-1} \prod_{j} x^{(i)}_{j}, \end{equation}
where the index $j$ runs over the entries of $Q^{(i)}$ with multiplicity.
    Let $\Lambda^r(\x):= \Lambda(\x^{(0)}) \otimes \cdots \otimes \Lambda(\x^{(r-1)})$ be
the ring of symmetric functions that are symmetric in each set of variables separately.

For each standard basis of the ring of symmetric functions $\Lambda$, we can naturally extend the definition to $\Lambda^r(\x)$:
\begin{itemize}
	\item $s_{\llambda}:= s_{\lambda^{(0)}}(\x^{(0)}) \cdots s_{\lambda^{(r-1)}}(\x^{(r-1)})$,
	\item $p_{\llambda}:= p_{\lambda^{(0)}}(\x^{(0)}) \cdots p_{\lambda^{(r-1)}}(\x^{(r-1)})$,
	\item $h_{\llambda}:= h_{\lambda^{(0)}}(\x^{(0)}) \cdots h_{\lambda^{(r-1)}}(\x^{(r-1)})$,
	\item $e_{\llambda}:= e_{\lambda^{(0)}}(\x^{(0)}) \cdots e_{\lambda^{(r-1)}}(\x^{(r-1)})$,
	\item $m_{\llambda}:= m_{\lambda^{(0)}}(\x^{(0)}) \cdots m_{\lambda^{(r-1)}}(\x^{(r-1)})$.
\end{itemize}
The set $\{s_{\llambda}: \llambda \vdash_r n\}$ forms a basis of $\Lambda^r(\x)_n$, the total-degree $n$ component of $\Lambda^r(\x)$.
Hence, the set $\{s_{\llambda}: \llambda \text{ an $r$-tuple of partitions}\}$ forms a basis of $\Lambda^r(\x)$.
Similarly, $p_{\llambda}, h_{\llambda}, e_{\llambda}, m_{\llambda}$ are bases.
The Schur functions may also be written as
\begin{equation}\label{eq:schur-ssyt}
	s_{\llambda}(\x)=\sum_{\QQ\in \SSYT^r(\llambda)}\x^\QQ.
\end{equation}

Equip $\Lambda^r(\x)$ with the Hall inner product $\langle -,- \rangle$ for which the $s_{\llambda}$ are orthonormal,
that is, $\langle s_{\llambda}, s_{\mmu} \rangle = \delta_{\llambda, \mmu}$. Consequently, $\langle h_{\llambda},
	m_{\mmu} \rangle = \delta_{\llambda, \mmu}$. For $f \in \Lambda^r(\x)$ we write $f^{\perp}$ for the adjoint with
respect to the Hall inner product of multiplication by $f$. The involution $\omega$ acts on each tensor factor, so
that $\omega s_{\llambda} = s_{\llambda'}$ with $\llambda':= (\lambda^{(0)\prime},\dots,\lambda^{(r-1)\prime})$.

For $0 \leq j \leq r-1$, let $U_j$ be the one-dimensional representation of $\Z_r$ on which $\zeta$ acts by
$\zeta^{-j}$. For $U$ a representation of $\Z_r$ and $S$ a representation of $\symm_m$, $U \wr S:= U^{\otimes m}
	\otimes S$ is a representation of $\Z_r \wr \symm_m$. For $\llambda \vdash_r n$, let
\begin{equation}
	V^{\llambda}:= \Ind_{\Z_r\wr\symm_{|\lambda^{(0)}|}\times\cdots\times\Z_r\wr\symm_{|\lambda^{(r-1)}|}}^{\Z_r\wr\mathfrak{S}_n}\big( (U_0 \wr S^{\lambda^{(0)}}) \boxtimes \cdots \boxtimes (U_{r-1} \wr S^{\lambda^{(r-1)}})\big),
\end{equation}
where $S^\mu$ denotes a Specht module of $\symm_{|\mu|}$.
The set $\{V^{\llambda}\}_{\llambda \vdash_r n}$ constitutes the complete set of irreducible representations of $\Z_r\wr\mathfrak{S}_n$ \cite[Chapter~I, Appendix~B]{Macdonald}.
The trivial representation is $V^{((n),\varnothing,\dots,\varnothing)}$ and the determinant representation is $V^{(\varnothing,\dots,\varnothing,(1^n))}$.

The \emph{Frobenius characteristic} of a $\Z_r\wr\mathfrak{S}_n$-module $V$ is obtained by applying the isomorphism that sends
each constituent irreducible module $V^{\llambda}$ to $s_{\llambda}$. 
That is, if $V \cong \bigoplus_{\llambda} c_{\llambda} V^{\llambda}$ for some $c_{\llambda} \in \Z_{\geq 0}$, and $V = \bigoplus_{a,b} V_{a,b}$ is bigraded, then
\begin{equation}
	\Frob(V):= \sum_{\llambda} c_{\llambda} s_{\llambda}, \qquad \grFrob(V;q;z):= \sum_{a,b} \Frob(V_{a,b}) q^a z^b.
\end{equation}

Since $\dim V^{\llambda} = \binom{n}{|\lambda^{(0)}|, \ldots, |\lambda^{(r-1)}|} \prod_{i} f^{\lambda^{(i)}}$, where
$f^{\mu} = |\SYT(\mu)|$, the Hilbert series of a bigraded $\Z_r \wr \mathfrak{S}_n$-module may be recovered from its
bigraded Frobenius series via
\begin{equation}\label{eq:Hilb-from-Frob}
	\Hilb(V;q;z) = \left\langle \grFrob(V;q;z), \big(h_1(\x^{(0)}) + \cdots + h_1(\x^{(r-1)})\big)^n \right\rangle.
\end{equation}

Induction from block-diagonal subgroups is multiplicative: for $\Z_r \wr \symm_{a} \times \Z_r \wr \symm_{b}$ and
representations $V$, $W$ of the two factors,
\begin{equation}\label{eq:Frob-induction}
	\Frob \left( \Ind_{\Z_r \wr \symm_{a} \times \Z_r \wr \symm_{b}}^{\Z_r \wr \symm_{a+b}} (V \boxtimes W) \right) = \Frob(V)\,\Frob(W),
\end{equation}
and likewise for more than two factors.

\begin{remark}\label{rmk:conventions}
	The labeling of the $U_j$ and the action of equation~\eqref{eq:action-on-variables} must be chosen together.
	The convention chosen in \cite{BR} uses $U_1,\ldots, U_r$ as the irreducible one-dimensional representations, where $U_j$ sends $\zeta$ to $\zeta^j$; their $U_j$ is our $U_{r-j}$. Consequently, any Frobenius series there can be turned into a Frobenius series in the conventions of this paper by the change of variables
	\begin{equation}
		\x^{(1)}\mapsto \x^{(r-1)}, \x^{(2)}\mapsto \x^{(r-2)},\ldots, \x^{(r)}\mapsto \x^{(0)}.
	\end{equation}
\end{remark}

It will be useful to define certain antisymmetrizing elements of the group algebra of $\Z_r\wr\mathfrak{S}_n$ indexed by
$r$-compositions. An \emph{$r$-composition} $\bbeta = (\beta^{(0)} \parallel \dots \parallel \beta^{(r-1)})$ of $n$ is an $r$-tuple of weak compositions with $|\beta^{(0)}| + \cdots + |\beta^{(r-1)}| = n$. 
We associate to $\bbeta$ the subgroup $\Z_r
	\wr \symm_{\bbeta} \subseteq \Z_r \wr \mathfrak{S}_n$ of block diagonal matrices
\begin{equation}
	\Z_r \wr \symm_{\bbeta}:=
	\left( \Z_r \wr \symm_{\beta^{(0)}_1}  \times \cdots \times \Z_r \wr \symm_{\beta_{\ell(\beta^{(0)})}^{(0)}}
	\right) \times \cdots \times \left(
	\Z_r \wr \symm_{\beta^{(r-1)}_1} \times \cdots \times \Z_r \wr \symm_{\beta_{\ell(\beta^{(r-1)})}^{(r-1)}} \right).
\end{equation}

For any integer $d$, we define a linear character $\sign^{(d)}_\bbeta:\Z_r\wr\symm_\bbeta\to \C^\times$ by the
following rule. Recall that $\chi(g)$ is the product of the nonzero entries in a monomial matrix $g$; $\sign(g)$ is the
sign of the permutation matrix underlying a monomial matrix $g$. We set
\begin{multline}\label{eq:sign_rep}
	\sign^{(d)}_{\bbeta}(g^{(0)}_1, \ldots,g^{(r-1)}_{\ell(\beta^{(r-1)})}):= \\
	\prod_{i=1}^{\ell(\beta^{(0)})} \sign(g^{(0)}_i) \chi(g^{(0)}_i)^d \times
	\prod_{i=1}^{\ell(\beta^{(1)})} \sign(g^{(1)}_i) \chi(g^{(1)}_i)^{d+1} \times \cdots \times\prod_{i=1}^{\ell(\beta^{(r-1)})} \sign(g^{(r-1)}_i) \chi(g^{(r-1)}_i)^{d+r-1}.
\end{multline}
This gives rise to the group algebra element
\begin{equation}\label{eq:eps_beta^d}
	\varepsilon^{(d)}_{\bbeta}:= \sum_{g \in \Z_r \wr \symm_{\bbeta}} \sign^{(d)}_{\bbeta}(g) g\in \C[\Z_r\wr\mathfrak{S}_n].
\end{equation}
For an $r$-composition
$\bbeta=(\beta^{(0)} \parallel \ldots \parallel \beta^{(r-1)})$,
we write
\begin{equation}
\begin{aligned}
	e_{\bbeta}(\x):= \prod_{c=0}^{r-1} \prod_{j=1}^{\ell(\beta^{(c)})} e_{\beta_j^{(c)}}(\x^{(c)}),\; 
    h_{\bbeta}(\x)
	:=
	\prod_{c=0}^{r-1}
	\prod_{j=1}^{\ell(\beta^{(c)})} 
	h_{\beta_j^{(c)}}(\x^{(c)}),
\end{aligned}
\end{equation}
with the convention $e_0=h_0=1$.
If $V$ is any $\Z_r \wr \mathfrak{S}_n$-module, then $\varepsilon^{(d)}_{\bbeta} \cdot V$ is a vector subspace of $V$. 
Define the {\em $d$-shift} of $\bbeta$ via
\begin{equation}\label{eq:d-shift}
	\bbeta^{\text{shift}_d}:= (\beta^{(-d)} \parallel \beta^{(1-d)} \parallel\cdots \parallel\beta^{(r-1-d)}),
\end{equation}
where superscripts are interpreted modulo $r$.

For $d=0$ and $\bbeta$ an $r$-partition, the following recovers \cite[Lemma~4.2]{BR} (with different conventions; see Remark~\ref{rmk:conventions}).

\begin{lemma}\label{lem:d-dimension-interpretation}
	For any $\Z_r \wr \mathfrak{S}_n$-module $V$, $r$-composition $\bbeta$ of $n$, and integer $d$,
	\begin{equation}\label{eq:dim-frob}
		\dim_\C (\varepsilon^{(d)}_{\bbeta} \cdot V) = \langle \Frob(V), e_{\bbeta^{\mathrm{shift}_d}}\rangle.
	\end{equation}
\end{lemma}

\begin{proof}
    Note that $\varepsilon^{(d)}_\bbeta=\sum_{h \in \Z_r\wr\symm_\bbeta}\sign^{(d)}_\bbeta(h) h$ is $|\Z_r\wr\symm_\bbeta|$ times the projection of $V$ onto the $\overline{\sign^{(d)}_\bbeta}$-isotypic component of its restriction to $\Z_r\wr\symm_\bbeta$.
    By Frobenius reciprocity (see \cite[Chapter I, Appendix B]{Macdonald}), 
    \begin{equation}
        \dim_\C(\varepsilon^{(d)}_\bbeta\cdot V) = \langle \Frob(V), \Frob(\Ind_{\Z_r\wr\symm_\bbeta}^{\Z_r\wr\symm_n} \overline{\sign^{(d)}_\bbeta} ) \rangle.
    \end{equation}
    On $\Z_r\wr\symm_{\beta^{(c)}_j}$, the character $\overline{\sign^{(d)}_\bbeta}$ maps $g \mapsto \sign(g)\zeta^{-(d+c)\eps(g)}$, which is the representation $U_{d+c}\wr S^{(1^{\beta^{(c)}_j})}$ ($d+c$ is read modulo $r$), whose Frobenius image is $e_{\beta^{(c)}_j}(\x^{(d+c)})$.
    Thus by equation~\eqref{eq:Frob-induction},
    \begin{equation}
        \Frob(\Ind_{\Z_r\wr\symm_\bbeta}^{\Z_r\wr\symm_n} \overline{\sign^{(d)}_\bbeta}) = \prod_{c,j}e_{\beta^{(c)}_j}(\x^{(d+c)})=e_{\bbeta^{\mathrm{shift}_d}}(\x).
    \end{equation}
\end{proof}

\subsection{Colored RSK}\label{subsec:colored-RSK}
Okada \cite{Okada} generalized the Robinson--Schensted correspondence for the symmetric group to wreath product groups $\Z_r \wr \mathfrak{S}_n$, by applying the Robinson--Schensted correspondence separately to the letters of each color.
We will need the analogous bijection for $r$-colored words, the \emph{$r$-colored Robinson--Schensted--Knuth (RSK) correspondence}, which is obtained in the same way from the classical RSK correspondence (see also \cite{baumann2008solomon} where a version is called the Robinson--Schensted--Okada correspondence).

Consider an $r$-colored word $g=g_1\cdots g_n$, where $g_k$ is the colored integer $w_k^{\eps_k}$. 
For each color $i\in \{0,\ldots, r-1\}$, let  $g_{a_1}g_{a_2}\cdots g_{a_m}$ ($a_1 < a_2 <\cdots < a_m$) be the subsequence of the letters of color $i$.
Apply RSK to the biword 
\begin{equation}
	\begin{pmatrix}
		a_1     & a_2     & \cdots & a_m     \\
		w_{a_1} & w_{a_2} & \cdots & w_{a_m}
	\end{pmatrix},
\end{equation}
to obtain a pair of Young tableaux $(P^{(i)},Q^{(i)})$ of the same shape. 
$P^{(i)}$ is semistandard with the bases $w_{a_1},\dots,w_{a_m}$ as entries, and $Q^{(i)}$ is standard with the positions $a_1,\dots,a_m$ as entries. 
Doing this for each $i$ yields a pair $(\PP(g),\QQ(g))$ of $r$-tableaux of the same shape where $\PP=(P^{(0)},\ldots, P^{(r-1)})$ and $\QQ=(Q^{(0)},\ldots, Q^{(r-1)})$. 
Recall the monomials $\x^{g}$ and $\x^{\QQ}$ of equation~\eqref{eq:word-and-tableau-monomials}.

\begin{theorem}[$r$-colored RSK]\label{thm:r-RSK}
	The \emph{$r$-colored RSK} map $g \mapsto (\PP(g),\QQ(g))$ is a bijection
	\begin{equation}
		\left\{
		\begin{array}{c}r\text{-colored words} \\
			\text{ of length $n$}
		\end{array}
		\right\}\to \bigsqcup_{\llambda\vdash_r n}\SSYT^r(\llambda)\times \SYT^r(\llambda).
	\end{equation}
    Moreover, for every $r$-colored word $g$ of length $n$,
	\begin{equation}\label{eq:RSK-monomial}
		\x^{g} = \x^{\PP(g)},
	\end{equation}
    \begin{equation}\label{eq:RSK-fDes}
        \fDes(g) = \fDes(\QQ(g)),
    \end{equation}
    and
    \begin{equation}\label{eq:RSK-fmaj}
        \fmaj(g) = \fmaj(\QQ(g)).
    \end{equation}
\end{theorem}

\begin{proof}
    The map is a bijection since it is invertible, and any pair in the codomain is attained. $\QQ(g)$ determines the positions $a_1 < \cdots < a_m$ of the letters of each color $i$, and then $(P^{(i)},Q^{(i)})$ determines the bases $w_{a_1},\dots,w_{a_m}$ by the inverse of RSK \cite{Knuth1970PERMUTATIONSMA}.
    The equality $\x^g = \x^{\PP(g)}$ holds because $P^{(i)}$ contains the bases of the letters of color $i$ with multiplicity.

Let $1 \leq k \leq n-1$. If $g_k$ and $g_{k+1}$ have the same color $i$, then they are consecutive letters of the subword of color $i$, so $k \in \Des(g)$ if and only if $w_k > w_{k+1}$, which holds if and only if $k+1$ lies in a lower row of $Q^{(i)}$ than $k$ (this is a descent in the typical sense within a tableau; see \cite[Section 7.23]{EC2}).
If they have different colors $i$ and $j$, then $g_k >_r g_{k+1}$ if and only if $i < j$, which is the condition for $k$ to be a descent of $\QQ(g)$ across distinct tableaux.
Thus $\Des(g) = \Des(\QQ(g))$.
Finally $0 \in \fDes(g)$ if and only if $g_1$ has nonzero color, if and only if $1 \notin Q^{(0)}$, if and only if $0 \in \fDes(\QQ(g))$.
Since the letters of color $j$ are recorded in $Q^{(j)}$, $\eps(g) = \sum_j j |Q^{(j)}|$ and $\fmaj(g) = \fmaj(\QQ(g))$. 
\end{proof}

\begin{example}
	Let $g = 3^1 2^0 2^0 4^2 3^2 1^0$ be the $3$-colored word of Example~\ref{ex:fdes}.
	The corresponding pair of $3$-tableaux is
	\begin{equation}
		P^{(0)} = \ytableaushort{12,2},\quad P^{(1)} = \ytableaushort{3}, \quad P^{(2)} = \ytableaushort{3,4},
	\end{equation}
	\begin{equation}
		Q^{(0)} = \ytableaushort{23,6}, \quad Q^{(1)} = \ytableaushort{1}, \quad
		Q^{(2)} = \ytableaushort{4,5}.
	\end{equation}
	The letters of color $0$ are $2^0,2^0,1^0$, at positions $2,3,6$, giving the biword 
    $\left(\begin{smallmatrix} 2 & 3 & 6 \\ 2 & 2 & 1\end{smallmatrix}\right)$, whose image under RSK is
$(P^{(0)},Q^{(0)})$ shown above.
	 The other colors are treated similarly.
	The $3$-tableau $\QQ=(Q^{(0)},Q^{(1)},Q^{(2)})$ is depicted in Figure~\ref{fig:recording}; its flag descents are $0,3,4$, and $\fmaj(\QQ(g)) = 3\cdot 7 + (0\cdot 3 + 1 \cdot 1 + 2 \cdot 2) = 26$, in agreement with Example~\ref{ex:fdes}. 
	\begin{figure}[H]
		\centering
		\begin{tikzpicture}
			\draw (0,0) rectangle (1,1);
			\draw (0,1) rectangle (1,2);
			\draw (1,2) rectangle (2,3);
			\draw (2,3) rectangle (3,4);
			\draw (2,4) rectangle (3,5);
			\draw (3,4) rectangle (4,5);
			\node at (1.5,3.5) {0};
			\node at (0.5,0.5) {5};
			\node at (0.5,1.5) {4};
			\node at (1.5,2.5) {1};
			\node at (2.5,3.5) {6};
			\node at (2.5,4.5) {2};
			\node at (3.5,4.5) {3};
		\end{tikzpicture}
		\caption{The recording tableau of the three-colored word $g = 3^1 2^0 2^0 4^2 3^2 1^0$.}
		\label{fig:recording}
	\end{figure}
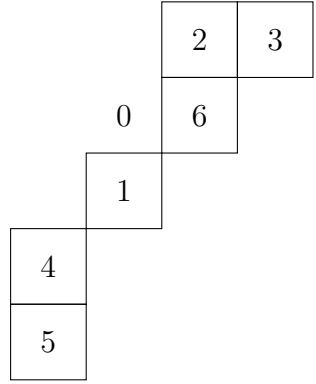
\end{example}

An $r$-colored permutation $g \in \Z_r\wr\mathfrak{S}_n$ is an $r$-colored word of length $n$ in which each of the integers $1,\ldots,n$ appears exactly once, and on $r$-colored permutations the correspondence reduces to that of Okada \cite[Theorem~3.3]{Okada}.

\begin{corollary}\label{cor:RS-colored-permutations}
	The $r$-colored RSK correspondence restricts to a bijection
	\begin{equation}\label{eq:RS-colored-permutations}
		\Z_r\wr\mathfrak{S}_n \to \bigsqcup_{\llambda\vdash_r n}\SYT^r(\llambda)\times \SYT^r(\llambda), \qquad
		g \mapsto (\PP(g),\QQ(g)),
	\end{equation}
	and for every $g \in \Z_r\wr\mathfrak{S}_n$ we have
	\begin{equation}\label{eq:RSK-fDes-fmaj}
		\fDes(g)=\fDes(\QQ(g)) \quad \text{and} \quad \fmaj(g)=\fmaj(\QQ(g)).
	\end{equation}
\end{corollary}

\begin{proof}
	By equation~\eqref{eq:RSK-monomial}, each of $1,\ldots,n$ appears exactly once as a base in $g$ if and only if each of $1,\ldots,n$ appears exactly once in $\PP(g)$ if and only if $\PP(g)$ is standard.
	Thus the bijection of Theorem~\ref{thm:r-RSK} restricts to equation~\eqref{eq:RS-colored-permutations}, and equation~\eqref{eq:RSK-fDes-fmaj} consists of equations~\eqref{eq:RSK-fmaj} and~\eqref{eq:RSK-fDes}.
\end{proof}

\subsection{Coinvariant rings}\label{subsec:coinvariants}
Fix $n \geq 1$. Fix a finite group $G \subseteq \GL_n$, and integers $k,j \geq 0$. Let $V$ be the defining
representation of $G$, so that $V \cong \C^n$. Place coordinates so that
\begin{equation}
	\Sym\left((V^*)^{\oplus k}\right) \otimes \bigwedge\left((V^*)^{\oplus j}\right) = \C[\bm{x}^{(1)},\ldots,\bm{x}^{(k)},\bm{\theta}^{(1)},\ldots,\bm{\theta}^{(j)}].
\end{equation}
Here, the $\bm{x}^{(i)} =
	(x^{(i)}_1,\ldots,x^{(i)}_n)$ are sets of commuting (bosonic) variables\footnote{Note that $\bm{x}^{(i)}$ denotes a set
	of variables of a ring, while $\x^{(i)}$ denotes an alphabet for the ring of symmetric functions $\Lambda^r$.} and the
$\bm{\theta}^{(i)} = (\theta^{(i)}_1,\ldots,\theta^{(i)}_n)$ are sets of anticommuting (fermionic) variables. Bosonic
variables commute with all variables, while fermionic variables anticommute with all fermionic variables. That is,
$x_j^{(i)}x_a^{(b)} = x_a^{(b)}x_j^{(i)}$, $x_j^{(i)}\theta_a^{(b)} = \theta_a^{(b)}x_j^{(i)}$, while
$\theta_j^{(i)}\theta_a^{(b)} = -\theta_a^{(b)}\theta_j^{(i)}$. This implies that $(\theta_j^{(i)})^2 = 0$. The
diagonal action of $G$ on $\C[\bm{x}^{(1)},\ldots,\bm{x}^{(k)},\bm{\theta}^{(1)},\ldots,\bm{\theta}^{(j)}]$ is given by
$G$ acting on each set of variables simultaneously, transforming like the coordinate functions on $V$ (for $G = \Z_r \wr \mathfrak{S}_n$, each set transforms as the coordinates in equation~\eqref{eq:action-on-variables}: $g \cdot x^{(a)}_i = \zeta^{-\eps(i)} x^{(a)}_{w(i)}$ and $g \cdot \theta^{(b)}_i = \zeta^{-\eps(i)} \theta^{(b)}_{w(i)}$). The \emph{$(k,j)$-bosonic-fermionic coinvariant ring} is
\begin{equation}\label{eq:coinvariant-ring-def}
	R^{(k,j)}_G:= \C[\bm{x}^{(1)},\ldots,\bm{x}^{(k)}, \bm{\theta}^{(1)},\ldots,\bm{\theta}^{(j)}]/I^{(k,j)}_G,
\end{equation}
where $I^{(k,j)}_G$ is the ideal generated by those diagonal $G$-invariants with no constant term.

In this definition, the variables span the dual $V^*$ of the defining representation $V$ of $G$.
When $G$ is generated by real reflections, such as for Weyl groups of types $A$ and $B$, we have that $V \cong V^*$, and it is inconsequential if $V$ or $V^*$ is used in the definition.
However, an important distinction arises for the wreath product groups $\Z_r \wr \mathfrak{S}_n$ with $r \geq 3$.
Griffeth \cite{Griffeth} studied the quotient of $\Sym(V) \otimes \Sym(V^*)$ by its diagonal invariants, which is not isomorphic to $R_{\Z_r \wr \mathfrak{S}_n}^{(2,0)}$ for $r \geq 3$.
Kim--Rhoades \cite{MR4381935} studied the quotient of $\bigwedge(V)
	\otimes \bigwedge(V^*)$ by its diagonal invariants, which is not isomorphic to $R_{\Z_r \wr \mathfrak{S}_n}^{(0,2)}$ for $r \geq
	3$.
Our main object of study is the superspace coinvariant ring $R_G^{(1,1)} = \C[\bm{x},\bm{\theta}]/I^{(1,1)}_G$, the quotient of the ring $\Sym(V^*) \otimes \bigwedge(V^*)$ of polynomial differential forms on $V$, as in \cite{SwansonWallach1,  SwansonWallach2, BR} (see \cite{SwansonWallach1} for some results for other choices regarding $V$ and $V^*$).
It is bigraded by $\bm{x}$-degree and $\bm{\theta}$-degree.

For $\mathbf{r} = (r_1,\ldots,r_k) \in \Z_{\geq 0}^k$ and $\mathbf{s} = (s_1,\ldots,s_j) \in \{0,1\}^j$, define the
\emph{polarized power sum} by
\begin{equation}
	p_{\mathbf{r};\mathbf{s}}:= \sum_{i=1}^{n} \big(x^{(1)}_i\big)^{r_1} \cdots \big(x^{(k)}_i\big)^{r_k} \big(\theta^{(1)}_i\big)^{s_1} \cdots \big(\theta^{(j)}_i\big)^{s_j},
\end{equation}
and write $|\mathbf{r}| := r_1 + \cdots + r_k$ and $|\mathbf{s}| := s_1 + \cdots + s_j$. Then $p_{\mathbf{r};\mathbf{s}}$ has degree $|\mathbf{r}| + |\mathbf{s}|$.

Orellana--Zabrocki \cite{OrellanaZabrocki} extended Weyl's polarization theorem \cite{Weyl} for $\mathfrak{S}_n$ from the bosonic case $j = 0$ to the bosonic-fermionic case.
This says that, for any $k$ and $j$, the polarized power sums with $1 \leq |\mathbf{r}| + |\mathbf{s}| \leq n$ generate the algebra of diagonal $\mathfrak{S}_n$-invariants.
The wreath product version easily follows.
\begin{proposition}\label{prop:Weyl-wreath-bf}
	The algebra of diagonal invariants
	$\C[\bm{x}^{(1)}, \ldots, \bm{x}^{(k)}, \bm{\theta}^{(1)}, \ldots, \bm{\theta}^{(j)}]^{\Z_r \wr \mathfrak{S}_n}$
	is generated by $1$ and the polarized power sums
	\begin{equation}\label{eq:gen-set-wreath}
		\left\{ p_{\bm{r};\bm{s}} \;\middle|\;
		\begin{array}{l}
			r \leq r_1 + \cdots + r_k +s_1 + \cdots +s_j \leq rn,         \\
			r_1 + \cdots + r_k +s_1 + \cdots +s_j\text{ a multiple of }r, \\
			r_1,\ldots,r_k \in \Z_{\geq 0}, s_1,\ldots, s_j \in \{0,1\}
		\end{array}
		\right\}.
	\end{equation}
	Consequently, the set~\eqref{eq:gen-set-wreath} generates the defining ideal $I_{\Z_r \wr \mathfrak{S}_n}^{(k,j)}$.
\end{proposition}
\begin{proof}
	Let $A_i$ be the subalgebra of $\C[x^{(1)}_i,\dots,x^{(k)}_i, \theta^{(1)}_i, \ldots, \theta^{(j)}_i]$ spanned by the monomials of degree divisible by $r$.
	Then $A_i$ is generated by those monomials of degree $r$.
	Consider that the invariants of the diagonal subgroup $\Z_r^n$ are spanned by the monomials in which, for each $i$, the total degree in the variables $x^{(1)}_i,\dots,x^{(k)}_i,\theta^{(1)}_i,\dots,\theta^{(j)}_i$ is divisible by $r$; hence it is $A_1 \otimes \cdots \otimes A_n$.
	Thus $\C[\bm{x}^{(1)},\dots,\bm{\theta}^{(j)}]^{\Z_r\wr\mathfrak{S}_n} = (A_1 \otimes \cdots \otimes A_n)^{\mathfrak{S}_n}$, where $w \in \mathfrak{S}_n$ maps $A_i$ onto $A_{w(i)}$.

	Let $\mu_{1,i}, \ldots, \mu_{N,i}$ be the monomials of degree $r$ in $A_i$.
	Let $B := \C[\bm{y}_1, \ldots, \bm{y}_N]$, where $\bm{y}_\ell = (y_{\ell,1}, \ldots, y_{\ell,n})$ is a set of bosonic variables if $\mu_\ell$ has an even number of $\theta$ factors and a set of fermionic variables otherwise, with $\mathfrak{S}_n$ acting diagonally.
	Then $y_{\ell,i} \mapsto \mu_{\ell,i}$ defines an $\mathfrak{S}_n$-equivariant surjection $\phi \colon B \to A_1 \otimes \cdots \otimes A_n$.
	Averaging over $\mathfrak{S}_n$ shows that $\phi$ restricts to a surjection $B^{\mathfrak{S}_n} \to (A_1 \otimes \cdots \otimes A_n)^{\mathfrak{S}_n}$.
	By Orellana--Zabrocki's extension of Weyl's theorem \cite{Weyl,OrellanaZabrocki} applied to $B$, $B^{\mathfrak{S}_n}$ is generated by the polarized power sums $\sum_{i=1}^n y_{1,i}^{c_1} \cdots y_{N,i}^{c_N}$ with $1 \leq c_1 + \cdots + c_N \leq n$, and $\phi$ maps such a polarized power sum to $\sum_{i=1}^n \mu_{1,i}^{c_1} \cdots \mu_{N,i}^{c_N}$, which is $0$ or $\pm p_{\mathbf{r};\mathbf{s}}$ for some $\mathbf{r}, \mathbf{s}$ with $|\mathbf{r}| + |\mathbf{s}| = r(c_1 + \cdots + c_N) \in \{r, 2r, \ldots, rn\}$.
	Thus the set~\eqref{eq:gen-set-wreath} generates the algebra of invariants.
	Since its elements have no constant term, it also generates $I_{\Z_r \wr \mathfrak{S}_n}^{(k,j)}$.
\end{proof}

Stembridge computed the graded Frobenius series of the coinvariant ring of $\Z_r \wr \mathfrak{S}_n$ and gave a combinatorial
formula in terms of $r$-tableaux. We write it here in our conventions.
\begin{theorem}[{\cite[Theorem 6.6]{Stembridge}}]\label{thm:stembridge}
	We have
	\begin{equation}\label{eq:stembridge}
		\grFrob(R^{(1,0)}_{\Z_r \wr \mathfrak{S}_n};q) = \sum_{\QQ \in \SYT^r_n} q^{\fmaj(\QQ)} s_{\shape(\QQ)}(\x).
	\end{equation}
\end{theorem}
Since the $\bm{\theta}$-degree zero part of $I^{(1,1)}_{\Z_r \wr \mathfrak{S}_n}$ is $I^{(1,0)}_{\Z_r \wr \mathfrak{S}_n}$, Theorem~\ref{thm:main-wreath} at $z = 0$ recovers Theorem~\ref{thm:stembridge}.

\subsection{Colored quasisymmetric functions}\label{subsec:colored-quasisymmetric}
The wreath product generalization of the classical quasisymmetric functions was introduced by Poirier
\cite{Poirier}, and subsequently, the theory has been further developed in \cite{HP2010,baumann2008solomon,
	moustakas2023descent, MR3629266, MantaciReutenauer}.
For the present work, we will only need the basic definitions and a few important properties of $r$-colored quasisymmetric functions, which we summarize below.
Our conventions differ from those of \cite{Poirier} by the relabeling $j   \mapsto r-j$ of the nonzero colors, which reverses their order (the color $0$ is largest in both). 

\begin{definition}\label{def: signed-set}
	Given fixed positive integers $n$ and $r$, a \emph{signed set} $\sigma$ is a pair $(T,\varepsilon)$ where $T\subseteq
		[n-1]$ and $\varepsilon$ is a map $[n]\to \{0,1,\ldots, r-1\}$, such that $\varepsilon(i)\ge \varepsilon(i+1)$ for every $i\in [n-1] \setminus T$. We let $S^r_n$ denote the collection of signed sets for $n$ and $r$.
\end{definition}

Next, we define a basis for the $r$-colored quasisymmetric functions indexed by $S^r_n$.
\begin{definition}
	For any signed set $\sigma=(T,\varepsilon)$, define the \emph{fundamental $r$-colored
		quasisymmetric function} as a formal power series on the alphabet $\x$, given by
	\begin{equation}
		F_{(T,\varepsilon)}(\x):=\sum_{\substack{i_1\le i_2\le \cdots\le i_n\\ k\in T\implies i_k<i_{k+1}}}x_{i_1}^{(\varepsilon (1))}x_{i_2}^{(\varepsilon(2))} \cdots x_{i_n}^{(\varepsilon(n))}.
	\end{equation}
\end{definition}

\begin{theorem}[{\cite[Corollary~9, Lemma~14]{Poirier}}]
	The set $\{F_{\sigma}(\x):\sigma\in S^r_n\}$ is linearly independent over $\C$, and contains $\Lambda^r(\x)_{n}$ in its span.
\end{theorem}

By \cite[Lemma~8]{Poirier}, the $F_{\sigma}$ (for $\sigma \in S^r_n$) are related by a unitriangular change of basis to Poirier's monomial colored quasisymmetric functions, so their span, which we denote by $\QSym^r_n$, is the degree $n$ component of Poirier's ring of $r$-colored quasisymmetric functions.
The $r$-colored quasisymmetric functions of \cite{HP2010} form a proper subspace of this in general. 

It is possible to obtain an explicit expansion of $s_{\llambda}(\x)$ in terms of the $F_{\sigma}$s. The following fact
is stated and proved in the $r=2$ case in \cite{MR3629266},
and as noted by
Moustakas \cite{moustakas2023descent}, easily extends to the general version stated below; we
include a proof adapted from \cite{MR3629266}. Recall that $\sDes(\QQ)$
was defined in Definition~\ref{def:fDes-sDes}. It is easy to see that $\sDes(\QQ)$ is a signed set in the sense of
Definition~\ref{def: signed-set}.

\begin{lemma}[{\cite[Proposition~3.1]{moustakas2023descent}}]
	\label{lem: schur-to-Qsym}
	For any $r$-partition $\llambda\vdash_rn$, the following identity holds:
	\begin{equation}
		s_{\llambda}(\x)=\sum_{\QQ\in \SYT^r(\llambda)} F_{\sDes(\QQ)}(\x).
	\end{equation}
\end{lemma}
\begin{proof}
	By equation~\eqref{eq:schur-ssyt}, the left hand side is
	\begin{equation}
		s_{\llambda}(\x)=\sum_{\UU\in \SSYT^r(\llambda)}\x^\UU.
	\end{equation}
	Given such a $\UU$, we standardize it as follows: start with the smallest element $i_1$ appearing in $\UU$, and replace all occurrences of it with the integers $1,2,\ldots$ in order, going from left to right within each tableau, and going through the tableaux in the order $U^{(r-1)}, U^{(r-2)}, \dots, U^{(0)}$.
	After that, replace all occurrences of the next smallest number $i_2$ with the next batch of integers in the same fashion.
	One may verify that this process yields an element of $\SYT^r(\llambda)$.

	Now given a $\QQ\in \SYT^r(\llambda)$, consider a possible $\UU\in \SSYT^r(\llambda)$ that standardizes to $\QQ$.
	Suppose $\sDes(\QQ)=(T,\varepsilon)$. If $i\in T$, then $i+1$ occurs either below $i$ on the same tableau, or in a
	different tableau to the left of $i$. In either case, they could not have been the same number in $\UU$. On the
	other hand, if $i\not\in T$, then it is possible that $i$ and $i+1$ have taken the place of two occurrences of the
	same element. This shows that the sum of monomials $\x^\UU$ over all such $\UU$ is precisely
	$F_{(T,\varepsilon)}(\x)$, which proves our claim.
\end{proof}

\subsection{\texorpdfstring{$r$}{r}-colored super ordered set partitions}
Chan--Rhoades \cite{ChanRhoades, Rhoades2018} studied the $r$-colored ordered set partitions in their study of the generalized
coinvariant algebras for wreath products. We will recall this definition and introduce some extensions that will be useful for us. In what follows, $r$ is a fixed positive integer, and $k, n$ are non-negative integers, unless otherwise stated.
\begin{definition}[\cite{ChanRhoades}]\label{def:OP_OmP}
	An \emph{$r$-colored ordered multiset partition} of size $n$ with $k$ blocks is an ordered sequence of $k$ non-empty sets consisting of $r$-colored positive integers $S_1 \mid \cdots \mid S_k$ whose cardinalities sum up to $n$. 
    An $r$-colored ordered multiset partition of size $n$ is called an \emph{$r$-colored ordered set partition} if every integer in $\{1,\ldots, n\}$ appears as a base exactly once.
	We denote the set of $r$-colored ordered set partitions of size $n$ with $k$ blocks by $\OP^r_{n,k}$, and the set of $r$-colored ordered multiset partitions of size $n$ with $k$ blocks by $\OmP^r_{n,k}$.
\end{definition}

The following are two example elements of $\OmP^3_{8,4}$:
\begin{equation}1^21^0\mid 3^24^12^1\mid 5^02^1\mid 2^0, \quad \text{and} \quad 3^1\mid 2^05^27^1\mid 8^2\mid 1^04^06^1, \end{equation}
where we have omitted the commas between elements of each block for brevity. Notice that the same base can occur
twice in the same block with different color, and also in different blocks with the same or different colors, but
not with the same color in the same block. The second example is also an element of $\OP^3_{8,4}$.

Sagan--Swanson \cite{SSreflection} introduced the colored super ordered set partitions. We will extend to multisets, with slightly different notation.
\begin{definition}[\cite{SSreflection}]\label{def:SOmP}
	An \emph{$r$-colored super ordered multiset partition} of size $n$ with $k$ blocks is a pair $\sigma=(Z, \rho)$ where $Z$ is a possibly empty set consisting of positive integers with \emph{nonzero} colors, and $\rho \in \OmP^r_{n-|Z|,k}$. 
    We call $Z$ the \emph{zero block} of $\sigma$. If $\rho=S_1\mid \cdots\mid S_k$, we will often write $\sigma$ as $\sigma=Z\mid S_1\mid \cdots\mid S_k$. 
    If all the integers in $\{1,\ldots, n\}$ appear as a base exactly once in $\sigma$, we will say it is an \emph{$r$-colored super ordered set partition}. The set of all $r$-colored super ordered set (resp. multiset) partitions of size $n$ with $k$ blocks will be denoted $\SOP^r_{n,k}$ (resp. $\SOmP^r_{n,k}$).
\end{definition}
For example, the following are all elements of $\SOmP^3_{8,4}$: \begin{equation}1^21^1\mid 3^2\mid 4^12^1\mid 5^02^1\mid 2^0, \quad \varnothing\mid1^21^0\mid 3^24^12^1\mid 5^02^1\mid 2^0,\quad \text{and}\quad 3^1\mid 2^05^27^1\mid 8^2\mid 1^0\mid 4^06^1, \end{equation}
whose zero blocks are $\{1^2,1^1\}$, $\varnothing$, and $\{3^1\}$ respectively. The last example is also an element
of $\SOP^3_{8,4}$. Notice that the zero block never contains an element of the form $i^0$.

\begin{definition}\label{def:OMP}
	An \emph{$r$-colored ordered multiset partition into multisets} of size $n$ with $k$ blocks is a sequence of non-empty multisets $S_1\mid\cdots\mid S_k$ consisting of $r$-colored positive integers, such that $|S_1| + \cdots + |S_k| = n$.
	We denote the set of all such $\rho$ by $\OMP_{n,k}^r$.
\end{definition}

\begin{definition}\label{def:SOMP}
	An \emph{$r$-colored super ordered multiset partition into multisets} (of size $n$, and $k$ blocks) is a pair $(Z,\rho)$ where $Z$ is a possibly empty multiset consisting of positive integers with nonzero colors, and $\rho \in \OMP_{n-|Z|,k}^r$.
	We denote the set of $r$-colored super ordered multiset partition into multisets of size $n$ with $k$ blocks by $\SOMP_{n,k}^r$.
\end{definition}

The above definitions differ from those of $\OmP^r_{n,k}$ and $\SOmP^r_{n,k}$ only in that colored integers are now
allowed to repeat, even within the same block. For example, \begin{equation}1^01^0\mid 3^2\mid 4^12^1\mid 5^02^1 2^0\in \OMP^{3}_{8,4},\quad 1^11^13^2\mid 3^2\mid 4^13^2\mid 2^12^1 2^0\in \SOMP^{3}_{9,3}\end{equation}where the last example has the multiset $\{1^11^13^2\}$ as its zero block.

For $\rho = S_1 \mid \cdots \mid S_k \in \OMP^r_{n,k}$ we write $\eps(\rho)$ for the sum of the colors of all entries
of $\rho$, counted with multiplicity.

\begin{remark}
	One can view $\OP$ and $\OmP$ as subsets of $\OMP$ by replacing multisets with sets, and view $\OMP$ as a subset of $\SOMP$ with empty \emph{zero block} $Z$.
	For some statistics, we take this approach and simply define a statistic on $\SOMP$, which yields a definition for $\OP$, $\OmP$, and $\OMP$.
\end{remark}
\begin{remark}
	We will use the notation $\OP^r_{n}$ to denote the set of $r$-colored ordered set partitions of size $n$ with any number of blocks: \begin{equation}
		\OP^r_n=\bigsqcup_{0\le k\le n}\OP^r_{n,k}.
	\end{equation}The notations $\SOP^r_n$, $\OmP^r_n$, $\SOmP^r_n$, $\OMP^r_n$, and $\SOMP^r_n$ are defined analogously.
\end{remark}
A \emph{descent-starred $r$-colored word} is a pair $(w,S)$ where $w$ is a word in the $r$-colored alphabet and $S
	\subseteq \Des(w)$. The $r=1$ (uncolored) version of this for ordered set-partitions was introduced by Remmel--Wilson
\cite{Remmel-Wilson} and then generalized to ordered multiset partitions by Wilson \cite{WilsonOMP} in their extensions
of MacMahon's equidistribution theorem. The colored version appears in \cite{ChanRhoades}.

\begin{definition}
	The \emph{reading word} of $\rho \in \OmP^r_{n,k}$ is the descent-starred $r$-colored word $(g,S)$ where $g$ is the $r$-colored word obtained by writing the elements of each block of $\rho$ in $>_r$-decreasing order, and then concatenating the blocks from left to right; $S$ is the set consisting of all $i\in [n-1]$ such that the $i$th and $(i+1)$st letters in $g$ appear in the same block of $\rho$.
\end{definition}

Let us illustrate the above definition with an example. Consider \begin{equation}\rho=1^21^0\mid 3^24^12^1\mid 2^15^0\mid 2^1\in
	\OmP^3_{8,4}.\end{equation}  In order to find its
reading word, we start by writing $\rho$ so that each block is in $>_r$-decreasing order:
\begin{equation}\rho=1^01^2\mid 4^12^13^2\mid 5^02^1\mid 2^1.\end{equation}
Now we concatenate the blocks, placing a star next to every entry that lies in the same block as the next entry:
\begin{equation}1^0_*1^2 4^1_*2^1_*3^2 5^0_*2^1 2^1.\end{equation}
This gives rise to the pair $(w,S)$ where $w$ is the word $1^01^2 4^12^13^2 5^02^1 2^1$ and $S=\{1,3,4,6\}$ records
the positions of the stars. Since each block was already arranged in decreasing $>_r$-order, $S\subseteq\Des(w)$.
Clearly, one can reconstruct $\rho$ given any descent-starred $r$-colored word $(w,S)$.

Next we define the major index of $\rho\in \OmP^r_{n,k}$, which is a slight modification of the $\maj$ statistic
defined in \cite{ChanRhoades}. Recall the $\fmaj$ statistic defined on $r$-colored words.

\begin{definition}\label{def:fmaj-OmP}
	Given $\rho \in \OmP^r_{n,k}$
	with reading word $(w,S)$, the \emph{flag-major index of $\rho$} is
	\begin{equation}
		\fmaj(\rho):= \fmaj(w) - \sum_{i \in S} \left( r \left| \Des(w) \cap [i, n-1] \right| - (r-1)\right).
	\end{equation}
\end{definition}

There are two equivalent ways to extend the statistic $\fmaj$ to $\SOmP^r_{n,k}$. The first method is to introduce the
\emph{flag-descent-starred $r$-colored words}, motivated by the following observation. Note that if we na\"ively form
the descent-starred $r$-colored words from elements of $\SOmP^r_{n,k}$, this leads to an ambiguity: $\sigma=1^11^2\mid
	4^12^13^2\mid 5^02^1\mid 2^1$ (whose zero block is $\{1^1,1^2\}$) and $\sigma'=\varnothing\mid1^11^2\mid 4^12^13^2\mid
	5^02^1\mid 2^1$ (whose zero block is empty) both have the same reading word. The solution is to introduce a star at the
$0$th position that denotes whether the first non-empty block is the zero block or not.

To this end define a \emph{flag-descent starred $r$-colored word} as a pair $(w,S)$ where $w$ is a word of length $n$ in the
$r$-colored alphabet and $S \subseteq \fDes(w)$. The following definition puts them in one-to-one
correspondence with $\SOmP^r_{n}$.

\begin{definition}\label{def:reading-word-SOmP}
	The \emph{reading word} of $\sigma \in \SOmP^r_{n,k}$ is the flag-descent-starred $r$-colored word $(g,S)$ where $g$ is the $r$-colored word obtained by writing the elements of each block of $\sigma$ in $>_r$-decreasing order, and then concatenating the blocks from left to right. $S$ is the set of indices $i\in [n-1]$ such that the $i$th and $(i+1)$st letters in $g$ appear in the same block of $\sigma$, along with $0$ if the zero block of $\sigma$ is non-empty.
\end{definition}

With this, the ambiguity is resolved:
\begin{equation}
	\begin{array}{lll}
		1^11^2\mid 4^12^13^2\mid 5^02^1\mid 2^1                 & \to {}^{\phantom{0}}_*1^1_*1^2 4^1_*2^1_*3^2 5^0_*2^1 2^1           & \to (1^11^2 4^12^13^2 5^02^1 2^1,\{0,1,3,4,6\}) \\
		\varnothing \mid1^11^2\mid 4^12^13^2\mid 5^02^1\mid 2^1 & \to \phantom{{}^{\phantom{0}}_*}1^1_*1^2 4^1_*2^1_*3^2 5^0_*2^1 2^1 & \to (1^11^2 4^12^13^2 5^02^1 2^1,\{1,3,4,6\}).
	\end{array}
\end{equation}As before, $S\subseteq \fDes(g)$ holds because each block is arranged in $>_r$-decreasing order, and the zero block
contains no element of color $0$. Now $\fmaj(\sigma)$ can be defined similarly.

\begin{definition}\label{def:fmaj-SOmP}
	Given $\sigma \in \SOmP^r_{n,k}$ with reading word $(w,S)$, the \emph{flag-major index of $\sigma$} is
	\begin{equation}
		\fmaj(\sigma):= \fmaj(w) - \sum_{i \in S} \left( r \left| \fDes(w) \cap [i, n-1] \right| - (r-1)\right).
	\end{equation}
\end{definition}

The second way to define $\fmaj$ on $\sigma=(Z,\rho)\in \SOmP^r_{n,k}$ is to compute it on $\rho$, and include a
correction factor to account for the zero block. For a zero block $Z=i_1^{c_1}i_2^{c_2}\cdots i_m^{c_m}$, define

\begin{equation}
	\varepsilon'(Z):=(c_1-1)+(c_2-1)+\cdots +(c_m-1).
\end{equation}

\begin{primeddef}{\ref*{def:fmaj-SOmP}\('\)}\label{def:fmaj-SOmP-primed}
	Given $\sigma=(Z,\rho)\in \SOmP^r_{n,k}$  with $\rho\in \OmP^r_{\ell,k}$, the \emph{flag-major index} of $\sigma$ is
	\begin{equation}
		\fmaj(\sigma):= \varepsilon'(Z)+\fmaj(\rho).
	\end{equation}
\end{primeddef}

Let us verify that Definitions~\ref{def:fmaj-SOmP} and~\ref{def:fmaj-SOmP-primed} are equivalent; this equivalence is used in Section~\ref{subsec:equidistribution}.

\begin{lemma}\label{lem:fmaj-equiv}
	Let $\sigma=(Z,\rho)\in \SOmP^r_{n,k}$ where $\rho\in \OmP^r_{\ell,k}$,  and let $(w,T)$ be the reading word of $\sigma$ (Definition~\ref{def:reading-word-SOmP}), so that $T\subseteq \fDes(w)$. Then \begin{equation}\varepsilon'(Z)+\fmaj(\rho)=\fmaj(w) - \sum_{i \in T} \left( r \left| \fDes(w) \cap [i, n-1] \right| - (r-1)\right).\end{equation}
\end{lemma}

\begin{proof}
	Let $(g,S)$ be the reading word of $\rho$. The claim is equivalent to
	\begin{multline}
		\varepsilon'(Z)+\varepsilon(g)+r\sum_{i\in \Des g}i- \sum_{i \in S}
		\left(r\,\#\left(\Des(g) \cap [i,\ell-1]\right)-(r-1)\right)\\ =\varepsilon(w)+r\sum_{i\in \fDes w}i- \sum_{i \in T}
		\left(r\,\#\left(\fDes(w) \cap [i,n-1]\right)-(r-1)\right).
	\end{multline}
	Let $m=|Z|$, $d=|\Des(g)|$, and let $\eta\in\{0,1\}$ indicate whether $m \in \fDes(w)$, that is, whether there is a descent between the last entry of $Z$ and the first entry of $g$. Set $\eta=0$ if $Z$ or $g$ is empty.
	The difference $\varepsilon(w)-(\varepsilon(g)+\varepsilon'(Z))$ counts $1$ for each element of $Z$, and thus equals $m$.

	Consider the difference $r\sum_{i\in \fDes w}i-r\sum_{i\in \Des g}i$.
	The first sum counts the positions $0,1,\ldots, m-1$ as flag descents, which contributes a surplus of $r(0+1+\cdots+(m-1))=r\binom{m}2$ to the sum.
	Further, each index $i\in \Des g$ gives rise to the index $i+m\in \fDes w$, contributing an additional $r\cdot dm$.
	Thus
	\begin{equation}
		r\sum_{i\in \fDes w}i-r\sum_{i\in \Des g}i=rdm+r\binom{m}{2} + rm\eta.
	\end{equation}

	Finally, consider the difference
	\begin{equation}
		\sum_{i \in T} \left(r\,\#\left(\fDes(w) \cap [i,n-1]\right)-(r-1)\right)-\sum_{i \in S}
		\left(r\,\#\left(\Des(g) \cap [i,\ell-1]\right)-(r-1)\right).\end{equation}

	Every element $i\in S$ contributes to a corresponding element $i+m\in T$, and these contributions cancel out. 
    The additional elements in $T$, the new stars, are $0,1,\ldots, m-1$. 
    For $0\leq t\leq m-1$, there are $m-t+d+\eta$ flag descents weakly to the right of $t$: the $m-t$ descents coming from the zero block, the $d$ shifted descents of $g$, and, when $\eta=1$, the boundary descent. Hence the contribution of these new stars is
	\begin{equation}
		\sum_{t=0}^{m-1} \left(r(m-t+d+\eta)-(r-1)\right) = rdm+r\binom{m}{2}+rm\eta+m.
	\end{equation}
	Combining these, the conclusion follows.
\end{proof}

For concreteness, let us compute these statistics on $\sigma=1^11^2\mid 4^02^03^1\mid 1^12^2\mid 2^2\in \SOmP^3_{8,3}$.
Its reading word is given by $(w,S)$ where \begin{equation}w=1^11^24^02^03^11^12^22^2,\quad S=\{0,1,3,4,6\}.\end{equation}The word $w$ has
descents at positions $1,3,4,5,6$, so $\fDes(w)=\{0,1,3,4,5,6\}$. Using Definition \ref{def:fmaj-SOmP}, we get
\begin{align*}
	\fmaj(\sigma)= & \varepsilon(w)+r\cdot\sum_{i\in \Des(w)}i-\sum_{i \in S} \left( r \left| \fDes(w) \cap [i, n-1] \right| - (r-1)\right) \\
	=              & (1+2+0+0+1+1+2+2)+3\cdot(1+3+4+5+6)                                                                                    \\
	               & -\left[(3\cdot 6-2)+(3\cdot 5-2)+(3\cdot 4-2)+(3\cdot 3-2)+(3\cdot 1-2)\right]                                         \\
	               & =19.
\end{align*}

Alternatively, we may view $\sigma$ as $(Z,\rho)$ where $Z=\{1^1,1^2\}$ and \begin{equation}\rho=4^02^03^1\mid 1^12^2\mid 2^2\in\OmP^3_{6,3},\end{equation}
with reading word $(4^02^03^11^12^22^2,\{1,2,4\})$. 
This word has descent set $\{1,2,3,4\}$, and so using
Definition~\ref{def:fmaj-OmP} we have
\begin{align*}
	\fmaj(\rho)= & (0+0+1+1+2+2)+3\cdot(1+2+3+4)             \\
	             & -[(3\cdot 4-2)+(3\cdot 3-2)+(3\cdot 1-2)] \\
	=            & 18
\end{align*}
which, using Definition~\ref{def:fmaj-SOmP-primed}, yields the matching answer $\fmaj(\sigma)=(1-1)+(2-1)+\fmaj(\rho)=19$.

The final set of definitions are two inversion-like statistics: in contrast with $\fmaj$, these will be defined on all
of $\SOMP^r_{n,k}$ rather than just $\SOmP^r_{n,k}$. We introduce an alternative ordering
\begin{equation}
	1^{r-1} \prec 1^{r-2} \prec \cdots \prec 1^0 \prec 2^{r-1} \prec \cdots \prec 2^0 \prec \cdots
\end{equation}
on the $r$-colored alphabet.
Simply put, this order compares colored integers first by the values of their bases (using the natural order on integers), and then uses color to break ties (higher color being smaller).
\begin{definition}
	For $\rho = S_1 \mid \cdots \mid S_k \in \OMP^r_{n,k}$, we write $\rho$ so that the entries in each block are sorted in decreasing $\succ$-order.

	A pair $(i,S_j)\in [n]\times \{S_1,\ldots, S_k\}$ is a \emph{weak (resp. strong) inversion} if $s_i$, the $i$th entry in $\rho$ from the left, appears in a block strictly to the left of $S_j$, and $s_i \succeq \min_{\prec} S_j$ (resp.
	$s_i \succ \min_{\prec} S_j$).
	We denote the set of weak (resp.
	strong) inversions of $\rho$ by $\wInv(\rho)$ (resp. $\sInv(\rho)$).
\end{definition}

\begin{definition}\label{def:sinv-winv}
	Let $\sigma=(Z,\rho) \in \SOMP^r_{n,k}$, where $\rho\in \OMP^r_{\ell,k}$, be given.
	The \emph{weak (resp.
		strong) inversion number} is defined by
	\begin{align}
		\winv(\sigma) & := r |\wInv(\rho)| + \eps(\rho) + (r-1) (\ell-k) + \eps'(Z), \\
		\sinv(\sigma) & := r |\sInv(\rho)| + \eps(\rho) + (r-1) (\ell-k) + \eps'(Z).
	\end{align}
\end{definition}

Observe that these statistics almost entirely ignore the zero block in their definitions, save for the
$\varepsilon'(Z)$ terms. The reader may compare this approach to Definition~\ref{def:fmaj-SOmP-primed} for $\fmaj$. As
an example, consider \begin{equation}\sigma=(Z,\rho)=3^21^11^1\mid 3^2\mid 4^13^2\mid 2^02^12^1\in \SOMP^3_{9,3},\end{equation} where $Z=3^21^11^1$ and $\rho=3^2\mid 4^13^2\mid 2^02^12^1\in \OMP^3_{6,3}$.
The weak inversion pairs are $(1,S_2)$, $(1,S_3)$, $(2,S_3)$, $(3,S_3)$. 
For instance, $(2,S_3)$ is a weak inversion pair because $4^1$ (the $2^{\text{nd}}$ entry from the left), lies to the left of $2^1$ (the $\prec$-minimal entry in $S_3$), and satisfies $4^1\succeq 2^1$. Only the last three of these pairs are also strong inversion pairs. Therefore
\begin{align*}
	\winv(\sigma)=3\cdot 4+(2+1+2+0+1+1)+2\cdot(6-3)+(1+0+0)=26, \\
	\sinv(\sigma)=3\cdot 3+(2+1+2+0+1+1)+2\cdot(6-3)+(1+0+0)=23.
\end{align*}
Note that the statistics $\winv$ and $\sinv$ coincide on $\SOP^r_{n,k}$.

\section{Main results for wreath product groups}\label{sec:main-results-wreath}

In this section, we prove our main result, Theorem~\ref{thm:main-wreath}, which states:
\begin{equation}\label{eq: wreath-poduct-Frob-formula}
	\grFrob(R^{(1,1)}_{\Z_r\wr\mathfrak{S}_n};q;z)=\sum_{\QQ\in \SYT^r_n}q^{\fmaj(\QQ)}\prod_{j=1}^{\fdes(\QQ)}\left(1+\frac{z}{q^{rj-r+1}}\right)s_{\shape(\QQ)}(\x).
\end{equation}
Our method is to show that both sides of equation~\eqref{eq: wreath-poduct-Frob-formula} satisfy the same recursion under the skewing operators $e_j(\x^{(i)})^\perp$, utilizing the following standard lemma (see for example \cite{ChanRhoades}).
\begin{lemma}\label{lem: perp-recursion-equivalence}
	Fix $n \geq 1$. For $F(\x),G(\x)\in\Lambda_n^r(\x)$ with coefficients in $\C(q)$, we have $F(\x)=G(\x)$ if and only if
	\begin{equation}
		e_j\left(\x^{(i)}\right)^\perp F(\x)=e_j\left(\x^{(i)}\right)^\perp G(\x)
	\end{equation}
	for all $j\ge 1$ and $0\le i\le r-1$.
\end{lemma}

Lemma~\ref{lem: perp-recursion-equivalence} holds because $\{ e_j(\x^{(i)}) : j \geq 1, \, \, 0 \leq i \leq r-1 \}$ is an algebraically independent generating set of $\Lambda^r(\x)$ and the Hall inner product is nondegenerate. 
In Section~\ref{subsec:recursion}, we derive the recursion satisfied by the algebraic left-hand side of equation~\eqref{eq: wreath-poduct-Frob-formula}.
In Sections~\ref{subsec:equidistribution}--\ref{subsec:recursion-ideal}, we show that the combinatorial right-hand side satisfies the same recursion. 

\subsection{Frobenius series recursion}\label{subsec:recursion}
We recall two results from \cite{BR} on the structure of $R^{(1,1)}_{\Z_r\wr\mathfrak{S}_n}$: a monomial basis, originally conjectured by Sagan--Swanson \cite{SSreflection}, and the ungraded module structure. 

For $J\subseteq[n]$, define the ordered product $\theta_J:=\prod_{j\in J}\theta_j$ with factors in increasing order of indices. 
The \emph{$(J,r)$-staircase} is the sequence $(\stair(J,r)_1,\ldots, \stair(J,r)_n)$ with entries defined as follows.
Let $[n]\setminus J=\{u_1<\cdots<u_{n-|J|}\}$ and $u_{n-|J|+1}:=n+1$.
Set $\stair(J,r)_{u_i}:=ir-1$, and for $j\in J$, set $\stair(J,r)_j:=ir-2$, where $i$ is the smallest index with $j<u_i$.

\begin{theorem}[\cite{BR}]\label{thm: monomial-basis}
	The collection of monomials
	\begin{equation}
		\{x_1^{a_1}\cdots x_n^{a_n}\theta_J:a_i\le\stair(J,r)_i, J\subseteq [n]\}
	\end{equation}
	descends to a basis of $R^{(1,1)}_{\Z_r\wr\mathfrak{S}_n}$.
\end{theorem}

The ungraded module structure of $R^{(1,1)}_{\Z_r\wr\mathfrak{S}_n}$ was described in \cite{BR} through a $\Z_r\wr\mathfrak{S}_n$-action on the $r$-ified ordered set partitions introduced there.
These are in bijection with $\SOP^r_n$: an element of $\SOP^r_n$ corresponds to the $r$-ified ordered set partition obtained by decreasing the color of each element of the zero block by $1$. 
We restate their results in terms of $\SOP^r_n$, transported through this bijection and the change of conventions of Remark~\ref{rmk:conventions}.

To describe the $\Z_r \wr \mathfrak{S}_n$-module structure on the $\C$-vector space $\C[\SOP^r_n]$, it suffices to give the actions of the subgroups $\mathfrak{S}_n$ and $\Z_r^n$ on the basis $\SOP^r_n$.

An element $w \in \mathfrak{S}_n$ acts on $\sigma\in \SOP^r_n$ by permuting the letters $\{1,\ldots, n\}$ while fixing the colors in place. 
For example, if $n=6$, $r=4$, $\sigma=(2^23^3\mid 1^05^36^2\mid 4^0)\in \SOP^4_6$, and $w=[1,3,6,5,4,2]$ in one-line notation, then
\begin{equation}
	w\cdot \sigma= (3^26^3\mid 1^04^32^2\mid 5^0)=(3^26^3\mid 1^02^24^3\mid 5^0).
\end{equation}
This induces an action of the subgroup $\mathfrak{S}_n$ of $\Z_r \wr \mathfrak{S}_n$ on $\C[\SOP^r_n]$ by linear extension.

An element $g=(\zeta^{c_1},\ldots, \zeta^{c_n})$ of the diagonal subgroup $\Z_r^n$ acts as follows. Consider $i^k$ appearing in $\sigma$.
If $i^k$ appears in a non-zero block of $\sigma$, then $i$ appears with the color $(k+c_i)\bmod r$ in $g\cdot \sigma$.
If $i^k$ appears in the zero block, then $i$ keeps the same color $k$, but $\sigma$ gets scaled by a factor of $\zeta^{(k-1)c_i}$.
For example, if $g=(\zeta^1,\zeta^3,\zeta^2,\zeta^1,\zeta^0,\zeta^2)$, then
\begin{equation}
	g\cdot\sigma=g\cdot(2^23^3\mid 1^05^36^2\mid 4^0)=\zeta^{(2-1)\cdot 3}\cdot\zeta^{(3-1)\cdot 2}\cdot(2^23^3\mid 1^15^36^0\mid 4^1).
\end{equation}

Given this action, we have the following result; $\det=\sign\cdot\chi$ is the determinant representation of Section~\ref{ss:rep_wr}.
\begin{theorem}[\cite{BR}]\label{thm: ungraded-iso}
	As ungraded $\Z_r\wr\mathfrak{S}_n$-modules,
	\begin{equation}
		R^{(1,1)}_{\Z_r\wr\mathfrak{S}_n}\cong\C[\SOP^r_n]\otimes\det.
	\end{equation}
\end{theorem}

Following Lemma~\ref{lem:d-dimension-interpretation}, we describe an explicit bihomogeneous basis of
$\varepsilon^{(d)}_\bbeta\cdot R^{(1,1)}_{\Z_r\wr\mathfrak{S}_n}$. 
For $d=0$ and $\bbeta$ an $r$-partition, this is \cite[Corollary 4.22]{BR}; the general case is a simple extension.

A \emph{signed $r$-composition} is a pair of $r$-compositions $(\bbeta,\ggamma)$ with $\beta^{(i)}\ge \gamma^{(i)}$ componentwise for each $0\le i\le r-1$; we write this as $\ggamma\le\bbeta$. 
For such a signed $r$-composition, let $B^{(i)}_j$ and $G^{(i)}_j$ be the partial sums
\begin{align}
	B^{(i)}_j&:= |\beta^{(0)}|+|\beta^{(1)}|+\cdots + |\beta^{(i-1)}|+\beta_1^{(i)}+\cdots +\beta_j^{(i)}, \label{eq:bigB}    \\
	 G^{(i)}_j&:= |\gamma^{(0)}|+|\gamma^{(1)}|+\cdots+|\gamma^{(i-1)}|+\gamma_1^{(i)}+\cdots +\gamma_j^{(i)}. \label{eq:bigG}
\end{align}
An $r$-composition $\bbeta$ of $n$ divides $[n] = \{1,\ldots,n\}$ into intervals of $\beta^{(i)}_j$ elements: we say the \emph{$\beta^{(i)}_j$-batch} of $[n]$ is \begin{equation}
	\{m\in \Z:B^{(i)}_j-\beta^{(i)}_j+1\le m\le B^{(i)}_j\}.
\end{equation}
For example, the 3-composition $\bbeta=(1,0,2\parallel 1\parallel 3,1)$ of $8$ divides $[8]$ into
\begin{equation}
	(\underbrace{1}_{\beta^{(0)}_1}\mid \underbrace{2,3}_{\beta^{(0)}_3}\parallel\underbrace{4}_{\beta^{(1)}_1}\parallel\underbrace{5,6,7}_{\beta^{(2)}_1}\mid \underbrace{8}_{\beta^{(2)}_2})
\end{equation}
where single bars separate batches of the same color and double bars separate colors.
In the example, the $\beta^{(0)}_2$-batch is $\varnothing$.

\begin{definition}
	For a signed $r$-composition $(\bbeta,\ggamma)$ of $n$, let $J(\bbeta,\ggamma) \subseteq [n]$ be the set
	\begin{equation}
		\bigsqcup_{\substack{0\le i\le r-1\\1\le j\le \ell(\beta^{(i)})}}\left\{m\in \Z:B^{(i)}_j-\gamma^{(i)}_j+1\le m\le B^{(i)}_j\right\}
	\end{equation}
    consisting of the largest $\gamma^{(i)}_j$ elements of each $\beta^{(i)}_j$-batch.
\end{definition}
For example, if $n = 8$, $r = 3$, $\bbeta = (1,0,2 \parallel 1 \parallel 3,1)$ and $\ggamma = (0, 0, 1 \parallel 1 \parallel 2, 0)$, then underlining the largest $\gamma_j^{(i)}$ numbers in each $\beta_j^{(i)}$-batch gives
\begin{equation}
	(1 \mid 2, \underline{3} \parallel \underline{4} \parallel 5,\underline{6},\underline{7} \mid 8),
\end{equation}
so $J(\bbeta,\ggamma) = \{ 3,4,6,7\}$.

\begin{definition}\label{def: shifted-antisym-basis}
	For a signed $r$-composition $(\bbeta, \ggamma)$ of $n$ and $0\le d\le r-1$, let $\AAA^{(d)}_{n,r}(\bbeta,\ggamma)$ be the set of all monomials $x_1^{a_1}\cdots x_n^{a_n}$ such that $a_k \leq \stair(J(\bbeta,\ggamma),r)_k$ for each $k= 1,\ldots,n$ and such that for each $0\le i\le r-1$ and for each $1\le j\le \ell(\beta^{(i)})$,
		      \begin{enumerate}
			      \item $a_{B^{(i)}_j-\beta^{(i)}_j+1},\ldots,a_{B^{(i)}_j-\gamma^{(i)}_j}$ is a \emph{strictly} increasing sequence of non-negative integers congruent to $d+i$ modulo $r$;
			      \item $a_{B^{(i)}_j-\gamma^{(i)}_j+1},\ldots,a_{B^{(i)}_j}$ is a \emph{weakly} increasing sequence of non-negative integers congruent to $d+i-1$ modulo $r$.
		      \end{enumerate}
	$B^{(i)}_j$ was defined in equation~\eqref{eq:bigB} as the last position corresponding to the part $\beta_j^{(i)}$.
\end{definition}

In our running example, $\stair(J(\bbeta,\ggamma),3)=\stair(\{3,4,6,7\},3)=(2,5,7,7,8,10,10,11) $, and for $d=1$, the set $\mathcal A^{(1)}_{8,3}(\bbeta,\ggamma)$ consists of the monomials $x_1^{a_1}\cdots x_8^{a_8}$ with
\begin{equation}
	(a_1,\ldots, a_8)\le (2,5,7,7,8,10,10,11)
\end{equation}
componentwise, and which satisfy:
$a_1\equiv a_2\equiv a_4\equiv 1$, $a_3\equiv a_5\equiv a_8\equiv 0$, and $a_6\equiv a_7\equiv 2 \pmod 3$ with $a_6\leq a_7$.

The next lemma shows that applying $\varepsilon^{(d)}_\bbeta$ to these monomials produces the desired basis; at $d=0$, this recovers \cite[Corollary~4.22]{BR}.
\begin{lemma}\label{lem:shifted-basis}
	For any $r$-composition $\bbeta$ of $n$, and any $0 \leq d \leq r-1$, the set of elements
	\begin{equation}\label{eq: shifted-basis}
		\bigsqcup_{\ggamma \leq \bbeta} \varepsilon^{(d)}_{\bbeta} \cdot (\AAA^{(d)}_{n,r}(\bbeta,\ggamma) \cdot \theta_{J(\bbeta,\ggamma)})
	\end{equation}
	descends to a $\C$-basis of $\varepsilon^{(d)}_{\bbeta} \cdot R^{(1,1)}_{\Z_r\wr\mathfrak{S}_n}$.
\end{lemma}

\begin{proof}
	The concept is: applying $\varepsilon^{(d)}_{\bbeta}$ to the monomial basis of Theorem~\ref{thm: monomial-basis} gives a spanning set, which reduces to the set in equation~\eqref{eq: shifted-basis} after discarding zeros and repeated scalar multiples (Step 1); counting that set (Step 2) and computing $\dim\varepsilon^{(d)}_{\bbeta}\cdot R^{(1,1)}_{\Z_r\wr\mathfrak{S}_n}$ from Theorem~\ref{thm: ungraded-iso} (Step 3) then shows that it is a basis.

	\paragraph{\emph{Step 1}}
	Let $m= x_1^{a_1}\cdots x_n^{a_n}\cdot \theta_J$ be a basis monomial of Theorem~\ref{thm: monomial-basis}, and let $\gamma^{(i)}_j$ be the number of elements of $J$ in the $\beta^{(i)}_j$-batch of $[n]$.
	Some permutation matrix $w\in \Z_r\wr\symm_\bbeta$ moves the elements of $J$ to the end of their batches and sorts the exponents within each of the two parts of each batch, so that $w\cdot m$ is, up to sign, a monomial $x_1^{b_1}\cdots x_n^{b_n}\cdot\theta_{J(\bbeta,\ggamma)}$ in which
	\begin{equation}
		b_{B^{(i)}_j-\beta^{(i)}_j+1},\ldots,b_{B^{(i)}_j-\gamma^{(i)}_j} \qquad \text{and} \qquad
		b_{B^{(i)}_j-\gamma^{(i)}_j+1},\ldots,b_{B^{(i)}_j}
	\end{equation}
	are weakly increasing for each batch.
	Moreover $b_k\leq\stair(J(\bbeta,\ggamma),r)_k$ for all $k$: moving the elements of $J$ to the end of their batch does not change the rank of any position outside $J$ and does not decrease the bound of any position in $J$, and sorting preserves the bounds, which are weakly increasing along each of the two parts of a batch.
	Since $\varepsilon^{(d)}_\bbeta\cdot m$ and $\varepsilon^{(d)}_\bbeta\cdot (w\cdot m)$ agree up to a scalar, we may restrict to monomials $m=x_1^{a_1}\cdots x_n^{a_n}\cdot\theta_{J(\bbeta,\ggamma)}$ of this form. 

     We will show that if
	\begin{enumerate}
		\item any of the exponents $a_{B^{(i)}_j-\beta^{(i)}_j+1},\ldots,a_{B^{(i)}_j-\gamma^{(i)}_j}$ are equal, or
		\item any of the exponents $a_{B^{(i)}_j-\beta^{(i)}_j+1},\ldots,a_{B^{(i)}_j-\gamma^{(i)}_j}$ is not congruent to $d+i$ modulo $r$, or
		\item any of the exponents $a_{B^{(i)}_j-\gamma^{(i)}_j+1},\ldots,a_{B^{(i)}_j}$ is not congruent to $d+i-1$ modulo $r$,
	\end{enumerate}
	then $\varepsilon^{(d)}_\bbeta\cdot m =0$.

    Suppose that $a_{k_1}=a_{k_2}$ for two indices $B^{(i)}_j-\beta^{(i)}_j+1 \le k_1<k_2\le B^{(i)}_j-\gamma^{(i)}_j$.
    The transposition $\tau=(k_1,k_2)$ lies in $\Z_r\wr\mathfrak S_\bbeta$, and $1-\tau \in \mathbb C[\Z_r\wr\mathfrak{S}_n]$ is a factor of $\varepsilon^{(d)}_\bbeta$ with $(1-\tau)\cdot m=m-m=0$, so $\varepsilon^{(d)}_\bbeta\cdot m=0$.

	Similarly, suppose that $a_{k}\not\equiv d+i\pmod r$ for some $B^{(i)}_j-\beta^{(i)}_j+1\le k\le B^{(i)}_j-\gamma^{(i)}_j$, and consider $g\in \Z_r^n$ with $\zeta$ as the $k$th entry and $1$ elsewhere. 
    Then
	\begin{equation}
		\varepsilon_g :=\sum_{j=0}^{r-1}\sign^{(d)}_{\bbeta}\left(g^j\right)\cdot g^j=\sum_{j=0}^{r-1}(\zeta^j)^{(d+i)}\cdot g^j\in \C[\Z_r\wr\mathfrak{S}_n]
	\end{equation}
	is a factor of $\varepsilon^{(d)}_\bbeta$.
	However, we have
	\begin{equation}
		\varepsilon_g\cdot m=\sum_{j=0}^{r-1}(\zeta^j)^{d+i}\cdot \zeta^{-ja_k}\cdot m = 0,
	\end{equation}
	since $d+i-a_k\not\equiv 0\pmod r $.

    The same argument with $g^j\cdot m=\zeta^{-j(a_k+1)}m$ shows that $\varepsilon^{(d)}_\bbeta\cdot m=0$ unless $a_k\equiv d+i-1\pmod r$ for $B^{(i)}_j-\gamma^{(i)}_j+1\le k\le B^{(i)}_j$.

	Thus the claimed set of elements spans $\varepsilon^{(d)}_{\bbeta}\cdot
		R^{(1,1)}_{\Z_r\wr\mathfrak{S}_n}$, implying that
	\begin{equation}\label{eq:shifted-bound}
		\sum_{\ggamma\le\bbeta}\#\AAA^{(d)}_{n,r}(\bbeta,\ggamma)\ge \dim\varepsilon_{\bbeta}^{(d)}\cdot R^{(1,1)}_{\Z_r\wr\mathfrak{S}_n}.
	\end{equation}

	\paragraph{\emph{Step 2}}
	It remains to show that equality holds in equation~\eqref{eq:shifted-bound}. We enumerate $\AAA_{n,r}^{(d)}(\bbeta,\ggamma)$.
	Fix $0\le i\le r-1$ and $1\le j\le \ell(\beta^{(i)})$.
	From the definition, we compute
	\begin{equation}
		\stair(J(\bbeta,\ggamma),r)_{B^{(i)}_j-\gamma^{(i)}_j}=r\cdot\left(B^{(i)}_j-G^{(i)}_j\right)-1.
	\end{equation}
    Since exactly $B^{(i)}_j-G^{(i)}_j$ integers in $[0,r(B^{(i)}_j-G^{(i)}_j)-1]$ are congruent to $d+i$ modulo $r$, the strictly increasing sequence $a_{B^{(i)}_j-\beta^{(i)}_j+1},\ldots,a_{B^{(i)}_j-\gamma^{(i)}_j}$ can be chosen in $\binom{B^{(i)}_j-G^{(i)}_j}{\beta^{(i)}_j-\gamma^{(i)}_j}$ ways.

    The positions $B^{(i)}_j-\gamma^{(i)}_j+1,\ldots,B^{(i)}_j$ lie in $J(\bbeta,\ggamma)$, and the next position outside $J(\bbeta,\ggamma)$ is the $(B^{(i)}_j-G^{(i)}_j+1)$st such position, so their staircase bound is $r(B^{(i)}_j-G^{(i)}_j)+r-2$.
    Exactly $B^{(i)}_j-G^{(i)}_j+1-\delta_{d+i\equiv_r 0}$ integers in $[0,r(B^{(i)}_j-G^{(i)}_j)+r-2]$ are congruent to $d+i-1$ modulo $r$, where $\delta_{d+i\equiv_r 0}$ is $1$ if $d+i\equiv 0\pmod r$ and $0$ otherwise.
    Hence the weakly increasing sequence $a_{B^{(i)}_j-\gamma^{(i)}_j+1},\ldots,a_{B^{(i)}_j}$ can be chosen in $\binom{B^{(i)}_j-G^{(i)}_j+\gamma^{(i)}_j-\delta_{d+i\equiv_r 0}}{\gamma^{(i)}_j}$ ways.

    Therefore
	\begin{equation}\label{eq:shifted-count}
		\#\AAA^{(d)}_{n,r}(\bbeta,\ggamma)=\prod_{i=0}^{r-1}\left(\prod_{j=1}^{\ell(\beta^{(i)})}\binom{B^{(i)}_j-G^{(i)}_j}{\beta^{(i)}_j-\gamma^{(i)}_{j}}\right.
		\left.\times\binom{B^{(i)}_j-G^{(i)}_j+\gamma^{(i)}_j-\delta_{d+i\equiv_r 0}}{\gamma^{(i)}_j}\right),
	\end{equation}
    where $\binom{-1}{0} = 1$.
	Summing this over all $\ggamma\le \bbeta$, we obtain an expression for the left-hand side of equation~\eqref{eq:shifted-bound}.

	\paragraph{\emph{Step 3}}
	We compute the right-hand side of equation~\eqref{eq:shifted-bound} using Theorem~\ref{thm: ungraded-iso}.
	We say that $\sigma\in \SOP^r_n$ is \emph{$(d,\bbeta)$-surviving} if:
	\begin{itemize}
		\item for each batch $\beta^{(i)}_j$ and for each $a<b$ in that batch, the block of $\sigma$ containing $a$ is the block containing $b$ or lies to its left (the zero block is the leftmost block);
		\item every element in a non-zero block has color $0$;
		\item every element in the zero block which is in the $\beta^{(i)}_j$-batch of $[n]$ has color $(-d-i)\bmod{r}$.
	\end{itemize}
    The final condition implies that if $d+i\equiv 0\pmod r$, then no element of the $\beta^{(i)}_j$-batch lies in the zero block.
	Let $v$ span the representation $\det $. 
    We claim that the set
	\begin{equation}\label{eq: OSP-basis}
		\{\varepsilon^{(d)}_{\bm\beta}\cdot (\sigma\otimes v):\sigma\text{ is a $(d,\bm\beta)$-surviving element of $\SOP^r_n$}\}
	\end{equation}
	is a basis for $\varepsilon^{(d)}_{\bm\beta}\cdot \left(\C[\SOP^r_n]\otimes \det \right)$.

	The elements $\varepsilon^{(d)}_{\bm\beta}\cdot (\sigma\otimes v)$ for $\sigma\in \SOP^r_n$ span $\varepsilon^{(d)}_{\bm\beta}\cdot\left(\C[\SOP^r_n]\otimes \det \right)$. 
    Given $\sigma$, some $w\in \Z_r\wr\symm_\bbeta$ arranges the elements of each batch in weakly increasing block order from left to right and gives every element in a non-zero block the color $0$.
    Since $\varepsilon^{(d)}_\bbeta\cdot(\sigma\otimes v)$ is a scalar multiple of $\varepsilon^{(d)}_\bbeta\cdot w\cdot(\sigma\otimes v)$, we may discard the $\sigma$ violating these conditions; the resulting subset still spans.

	Next, suppose that $k^c$ lies in the zero block of $\sigma$, with $k$ in a $\beta^{(i)}_{*}$-batch. 
    Let $g\in\Z_r^n$ and $\varepsilon_g$ be as in Step 1.
    Then
	\begin{equation}
		\varepsilon_g\cdot(\sigma\otimes v)=\sum_{j=0}^{r-1}\zeta^{j(d+i)}\cdot (\zeta^{(c-1)j}\cdot \sigma)\otimes(\zeta^j\cdot v)=\sum_{j=0}^{r-1}(\zeta^{d+i+c})^j(\sigma\otimes v)
	\end{equation}
	vanishes unless $c\equiv -(d+i)\pmod r$, so we may also discard the $\sigma$ violating the third condition.
    Thus the set in equation~\eqref{eq: OSP-basis} spans.

	Now let $\sigma$ be $(d,\bbeta)$-surviving. 
    If $g \in \Z_r\wr\symm_\bbeta$ maps $\sigma$ to a scalar multiple $c_g \sigma$, then $g$ only permutes elements of a common batch lying in a common block of $\sigma$ and changes colors only in the zero block. 
    Checking these two kinds of generators shows that $\sign^{(d)}_\bbeta(g) c_g \det(g)=1$.
	Thus the coefficient of $\sigma\otimes v$ in $\varepsilon^{(d)}_\bbeta\cdot(\sigma\otimes v)$ is the number of such $g$, and in particular, $\varepsilon^{(d)}_\bbeta\cdot(\sigma\otimes v)\neq 0$.
	Since $\varepsilon^{(d)}_\bbeta\cdot(\sigma\otimes v)$ is supported on the $\Z_r\wr\symm_\bbeta$-orbit of $\sigma$, and each orbit contains at most one $(d,\bbeta)$-surviving element, the elements of equation~\eqref{eq: OSP-basis} have pairwise disjoint supports, so they are linearly independent, proving the claim.

    It only remains to enumerate the $(d,\bbeta)$-surviving elements of $\SOP^r_n$, which we do by adapting \cite[Lemma~4.7]{BR}: every such element is built from the empty set partition by inserting the $\beta^{(i)}_j$-batches of numbers in order.

    For example, with $\bbeta = (1,0,2 \parallel 1 \parallel 3,1)$, and $d=2$, one $(d,\bbeta)$-surviving $\sigma$ is built as follows, where each line shows the batch about to be inserted and the partition built so far:
	\begin{align*}
		\{1\}\to       & (\varnothing)                                 \\
		\varnothing\to & (\varnothing\mid 1^0)                         \\
		\{2,3\}\to     & (\varnothing\mid 1^0)                         \\
		\{4\}\to       & (2^1\mid 1^0,3^0)                             \\
		\{5,6,7\}\to   & (2^1\mid 1^0,3^0,4^0)                         \\
		\{8\}\to       & (2^1, 5^2\mid 6^0\mid 1^0,3^0,4^0,7^0)        \\
		\sigma=        & (2^1,5^2\mid 6^0\mid 8^0\mid 1^0,3^0,4^0,7^0).
	\end{align*}
	The colors are forced by the definition of a $(d,\bbeta)$-surviving element, so it suffices to count the ways to insert each batch.

	Consider inserting the $\beta^{(i)}_j$-batch, which creates at
	most $\beta^{(i)}_j$ new blocks. 
    Define a sequence of numbers $\gamma^{(i)}_j$ so that exactly $\beta^{(i)}_j-\gamma^{(i)}_j$ new non-zero blocks are created; then $(\bbeta,\ggamma)$ is a signed composition.
	Before this insertion, the partition has $p:=B^{(i)}_j-G^{(i)}_j-(\beta^{(i)}_j-\gamma^{(i)}_j)$ non-zero blocks. 
    Insert $\beta^{(i)}_j-\gamma^{(i)}_j$ new non-zero blocks, each consisting of a single placeholder $\bullet$, in one of
	\begin{equation}
		\binom{p+\beta^{(i)}_j-\gamma^{(i)}_j}{\beta^{(i)}_j-\gamma^{(i)}_j}=\binom{B^{(i)}_j-G^{(i)}_j}{\beta^{(i)}_j-\gamma^{(i)}_j}
	\end{equation}
	ways. 
    Then insert $\gamma^{(i)}_j$ placeholders $\circ$ (possibly several into one block) into the $p+\beta^{(i)}_j-\gamma^{(i)}_j$ non-zero blocks (both old and new), and into the zero block if $d+i\not\equiv 0\pmod r$. This is a choice of a multiset of size $\gamma^{(i)}_j$ from $p+\beta^{(i)}_j-\gamma^{(i)}_j+1-\delta_{d+i\equiv_r 0}$ blocks, which can be made in
	\begin{equation}
		\binom{p+\beta^{(i)}_j-\delta_{d+i\equiv_r 0}}{\gamma^{(i)}_j}=\binom{B^{(i)}_j-G^{(i)}_j+\gamma^{(i)}_j-\delta_{d+i\equiv_r 0}}{\gamma^{(i)}_j}
	\end{equation}
	ways. 
    Finally, replace all $\bullet$ and $\circ$ symbols from left to right by the elements of the $\beta^{(i)}_j$-batch in increasing order, with color $0$ in the non-zero blocks and color $(-d-i) \bmod r$ in the zero block.

    For example, inserting the $\beta^{(2)}_1$-batch $\{5,6,7\}$ into $(2^1\mid 1^0,3^0,4^0)$ with $\gamma^{(2)}_1=2$ includes the following two (among many) possibilities:
	\begin{equation}
		\begin{array}{l}
			(2^1\mid 1^0,3^0,4^0)\to (2^1\mid\bullet\mid 1^0,3^0,4^0)\to (2^1, \circ\mid\bullet\mid 1^0,3^0,4^0,\circ)\to (2^1, 5^2\mid  6^0\mid 1^0,3^0,4^0,7^0) \\
			(2^1\mid 1^0,3^0,4^0)\to (2^1\mid 1^0,3^0,4^0\mid\bullet)\to(2^1\mid 1^0,3^0,4^0,\circ,\circ\mid \bullet)\to (2^1\mid 1^0,3^0,4^0,5^0,6^0\mid 7^0)
		\end{array}.
	\end{equation}
	Different $\gamma^{(i)}_j$ yield different $\sigma$; the number of possible $\sigma$ that can be constructed is
	\begin{equation}
		\sum_{\ggamma\le \bbeta}\left(\prod_{i=0}^{r-1}\prod_{j=1}^{\ell(\beta^{(i)})}\binom{B^{(i)}_j-G^{(i)}_j}{\beta^{(i)}_j-\gamma^{(i)}_{j}}
		\times\binom{B^{(i)}_j-G^{(i)}_j+\gamma^{(i)}_j-\delta_{d+i\equiv_r 0}}{\gamma^{(i)}_j}\right)
	\end{equation}
	which, by equation~\eqref{eq:shifted-count}, agrees with the left-hand side of equation~\eqref{eq:shifted-bound}.
	Thus equality holds in equation~\eqref{eq:shifted-bound}, and the lemma is proven.
\end{proof}

We now derive the recursion for the bigraded Frobenius series, individually for each $z$-graded piece. 
Define $C_{n,k}(\x,q)$ by
\begin{equation}
	\grFrob(R^{(1,1)}_{\Z_r\wr\mathfrak{S}_n};q;z)=\sum_{k=0}^nz^{n-k} C_{n,k}(\x,q),
\end{equation}
with $C_{0,0}=1$ and $C_{m,s}=0$ unless $0\le s\le m$.
For $k\ge0$ and $0\le i\le r-1$, define the finite alphabets
\begin{equation}
	U_{k,i}:=(q^{i+ra})_{0\le a<k},\qquad V_{k,i}:=
	\begin{cases}
		(q^{r-1+ra})_{0\le a<k},    & i=0,           \\
		(q^{i-1+ra})_{0\le a\le k}, & 1\le i\le r-1.
	\end{cases}
\end{equation}
For $0\le\ell\le j$, put
\begin{equation}\label{eq:common-branching-coefficient}
	\kappa^{(i)}_{k;j,\ell}
	:=e_\ell(U_{k,i})h_{j-\ell}(V_{k,i}),
\end{equation}
where $e_0=h_0=1$, and positive-degree elementary and complete homogeneous symmetric functions vanish on the empty alphabet. 
Explicitly, for $j\ge1$ and $0\le\ell\le\min(j,k)$, 
\begin{equation}\label{eq: finite-alphabet-eval}
	\begin{aligned}
		\kappa^{(0)}_{k;j,\ell}
		 & =q^{r\binom{\ell}{2}+(j-\ell)(r-1)}
		\qbinom{k}{\ell}_{q^r}
		\qbinom{k+j-\ell-1}{j-\ell}_{q^r},     \\
		\kappa^{(i)}_{k;j,\ell}
		 & =q^{r\binom{\ell}{2}+j(i-1)+\ell}
		\qbinom{k}{\ell}_{q^r}
		\qbinom{k+j-\ell}{j-\ell}_{q^r}
		\qquad(1\le i\le r-1).
	\end{aligned}\end{equation}

\begin{proposition}\label{lem: Cnk-recursion}
	For $0\le k\le n$, $0\le i\le r-1$, and $1\le j\le n$,
	\begin{equation}
		e_j\bigl(\x^{(i)}\bigr)^\perp C_{n,k}
		=\sum_{\ell=0}^{\min(j,k)}
		\kappa^{(i)}_{k;j,\ell}\,C_{n-j,k-\ell}.
	\end{equation}
\end{proposition}

\begin{proof}
	Fix $0\le i\le r-1$ and choose $d$ so that $d-i\equiv 1\pmod r$.
	By equation~\eqref{eq:d-shift}, the $(r-1)$st component of an $r$-composition becomes the $i$th component after the $d$-shift.
	Let $\bbeta$ be an $r$-composition of $n$ whose component $\beta^{(r-1)}$ has last part $j$, and let $\bbeta_0$ be obtained from $\bbeta$ by deleting that part; every $r$-composition $\bbeta_0$ of $n-j$ arises in this way.
	Then $e_{\bbeta^{\operatorname{shift}_d}}(\x) = e_j(\x^{(i)}) e_{\bbeta_0^{\operatorname{shift}_d}}(\x)$, so	\begin{equation}\label{eq:C-recursion-pairing}
		\left\langle e_j(\x^{(i)})^\perp C_{n,k}, e_{\bbeta_0^{\operatorname{shift}_d}} \right\rangle = \left\langle C_{n,k}, e_{\bbeta^{\operatorname{shift}_d}} \right\rangle.
	\end{equation}
	By Lemmas~\ref{lem:d-dimension-interpretation} and
	\ref{lem:shifted-basis}, the right-hand side is the graded generating function
	\begin{equation}\label{eq:C-recursion-basis-sum}
		\sum_{\substack{\ggamma\le\bbeta\\|\ggamma|=n-k}} \sum_{m\in \AAA^{(d)}_{n,r}(\bbeta,\ggamma)} q^{\deg m}.
	\end{equation}
	Let the last part of $\ggamma$ be $j-\ell$, and let $(\bbeta_0,\ggamma_0)$ be obtained by deleting the last parts of $\bbeta$ and $\ggamma$, so that $|\bbeta_0| = n-j$ and $|\ggamma_0| = n-k -(j-\ell)$.
	Write a monomial
	\begin{equation}\label{eq:mon_beta_gamma}
		m = x_1^{a_1} \cdots x_n^{a_n} \in \AAA^{(d)}_{n,r}(\bbeta,\ggamma)
	\end{equation}
	as $m = m_0 m_1$ with $m_0=x_1^{a_1}\cdots x_{n-j}^{a_{n-j}}$ and $m_1=x_{n-j+1}^{a_{n-j+1}}\cdots x_n^{a_n}$.
    Since $\stair(J(\bbeta,\ggamma),r)$ and $\stair(J(\bbeta_0,\ggamma_0),r)$ agree on $[n-j]$, we have $m_0\in\AAA^{(d)}_{n-j,r}(\bbeta_0,\ggamma_0)$.
	Conversely, every monomial in $\AAA^{(d)}_{n-j,r}(\bbeta_0,\ggamma_0)$ extends to a monomial in $\AAA^{(d)}_{n,r}(\bbeta,\ggamma)$ by any choice of exponents for the last $j$ variables satisfying the conditions of Definition~\ref{def: shifted-antisym-basis}.
	Thus, for fixed $\ell$, the corresponding part of equation~\eqref{eq:C-recursion-basis-sum} factors as
	\begin{equation}
		\left( \sum_{\substack{\ggamma_0\leq\bbeta_0\\|\ggamma_0|=(n-j)-(k-\ell)}} \sum_{m_0\in\AAA^{(d)}_{n-j,r}(\bbeta_0,\ggamma_0)} q^{\deg(m_0)} \right) \left( \sum_{m_1}q^{\deg(m_1)} \right).
	\end{equation}
	By Lemmas~\ref{lem:d-dimension-interpretation} and~\ref{lem:shifted-basis}, the first factor is $\langle C_{n-j,k-\ell}, e_{\bbeta_0^{\operatorname{shift}_d}}(\x) \rangle$.
	It remains to compute the second factor.

    By Definition~\ref{def: shifted-antisym-basis}, the exponents $a_{n-j+1},\ldots,a_{n-j+\ell}$ are strictly increasing and congruent to $i$ modulo $r$, since $d+r-1\equiv i\pmod r$.
	Since $|\ggamma_0|=n-k-(j-\ell)$, the positions $n-j+1,\ldots,n-j+\ell$ are the $(k-\ell+1)$st through $k$th positions outside $J(\bbeta,\ggamma)$, so their staircase bounds are $r(k-\ell+1)-1,\ldots,rk-1$, and the allowed values are $i,i+r,i+2r,\ldots,i+(k-1)r$.
	Their degree generating function is
	\begin{equation}\label{eq:strict-e-generating-function}
		e_\ell \left( q^i,q^{i+r},\ldots,q^{i+(k-1)r} \right).
	\end{equation}
	By Definition~\ref{def: shifted-antisym-basis}, the remaining exponents $a_{n-j+\ell+1},\ldots,a_n$ in the monomial are weakly increasing and congruent to $i-1$ modulo $r$ since $d+(r-1)-1 \equiv i-1\pmod r$.
    Their positions lie in $J(\bbeta,\ggamma)$ and are followed by no position outside it, so each has staircase bound $r(k+1)-2$, as $|J(\bbeta,\ggamma)|=n-k$.

	If $1\le i\le r-1$, the allowed values are $i-1,i-1+r,\ldots,i-1+kr$, while if $i=0$, the allowed values are  $r-1,2r-1,\ldots,kr-1$. 
    In both cases these form the alphabet $V_{k,i}$, so the generating function of the weakly increasing sequence is $h_{j-\ell}(V_{k,i})$.
    Thus the second factor is $\kappa^{(i)}_{k;j,\ell}$.
	Summing over $\ell$ and using equation~\eqref{eq:C-recursion-pairing}, we obtain
	\begin{equation}
		\left\langle e_j(\x^{(i)})^\perp C_{n,k}, e_{\bbeta_0^{\operatorname{shift}_d}}(\x) \right\rangle = \left\langle \sum_{\ell=0}^{\min(j,k)} \kappa^{(i)}_{k;j,\ell} C_{n-j,k-\ell}, e_{\bbeta_0^{\operatorname{shift}_d}}(\x) \right\rangle.
	\end{equation}
	As $\bbeta_0$ varies, the functions $e_{\bbeta_0^{\operatorname{shift}_d}}$ span $\Lambda^r_{n-j}(\x)$, and the recursion follows.
\end{proof}

\subsection{Equidistribution of two statistics}\label{subsec:equidistribution} Define $D_{n,k}(\x,q) \in \Lambda^r(\x)$ by
\begin{equation}
	\sum_{\QQ\in \SYT^r_n}q^{\fmaj(\QQ)}\prod_{j=1}^{\fdes(\QQ)}\left(1+\frac{z}{q^{rj-r+1}}\right)s_{\shape(\QQ)}(\x)=\sum_{k=0}^n z^{n-k} D_{n,k}(\x,q).
\end{equation}
We will show that $D_{n,k}$ satisfies the same recursion as the function $C_{n,k}$ in Proposition~\ref{lem: Cnk-recursion} in four steps:
\begin{itemize}
	\item Lemma~\ref{lem: Dnk-to-fmaj} rewrites $D_{n,k}(\x,q)$ as the generating function of $\fmaj$ over $\SOmP^r_{n,k}$;
	\item Lemma~\ref{lem: fmaj-to-sinv} exchanges $\fmaj$ for $\sinv$ using an equidistribution theorem of Wilson \cite{WilsonOMP}; 
	\item Lemma~\ref{lem: omega-sinv-to-winv} computes $\omega D_{n,k}$ as the generating function of the statistic $\winv$ over $\SOMP^r_{n,k}$;
	\item Lemma~\ref{lem: Ink-recursion} derives the recursion for $h_j(\x^{(i)})^\perp$ applied to that generating function.
\end{itemize}
For $\rho\in\SOMP^r_{n,k}$, let $\x^\rho$ be the monomial in which the exponent of $x^{(i)}_j$ is the number of times $j^i$ appears in $\rho$.

\begin{lemma}\label{lem: Dnk-to-fmaj}
	We have
	\begin{equation}
		D_{n,k}(\x,q)=\sum_{\rho\in\SOmP^r_{n,k}}q^{\fmaj(\rho)}\x^\rho.
	\end{equation}
\end{lemma}
\begin{proof}
	Expand $s_{\llambda}(\x)=\sum_{\UU\in \SSYT^r(\llambda)}\x^\UU$ in the definition of $D_{n,k}$:
	\begin{align*}
		\sum_{k=0}^n{z^{n-k}}
		D_{n,k}(\x,q) & =\sum_{\QQ\in \SYT^r_n}q^{\fmaj(\QQ)}\prod_{j=1}^{\fdes(\QQ)}\left(1+\frac{z}{q^{rj-r+1}}\right)\sum_{\substack{\UU \in\SSYT^r_n\\ \shape(\UU)=\shape(\QQ)}}\x^\UU \\
		              & =\sum_{\substack{\UU\in \SSYT^r_n,\QQ\in\SYT^r_n\\ \shape(\UU)=\shape(\QQ)}}q^{\fmaj(\QQ)}\prod_{j=1}^{\fdes(\QQ)}\left(1+\frac{z}{q^{rj-r+1}}\right)\x^\UU.
	\end{align*}
	Apply the colored RSK correspondence (Theorem~\ref{thm:r-RSK}), a bijection $w\mapsto (\UU,\QQ)$ from $\N_r^n$ to the pairs $(\UU,\QQ)\in \SSYT^r_n\times \SYT^r_n$ of the same shape, with $\fmaj(w)=\fmaj(\QQ)$, $\fDes(w)=\fDes(\QQ)$, and $\x^w=\x^\UU$.
    Thus
	\begin{align*}
		\sum_{k=0}^n z^{n-k}
		D_{n,k}(\x,q) & =\sum_{w\in\N_r^n}q^{\fmaj(w)}\prod_{j=1}^{\fdes(w)}\left(1+\frac{z}{q^{rj-r+1}}\right)\x^w                                                        \\
		              & =\sum_{w\in \N_r^n}q^{\fmaj(w)}\sum_{T\subseteq\fDes(w)}z^{\#T}q^{-\sum_{i\in T}\left(r\,\#\left(\fDes(w) \cap [i,n-1]\right)-(r-1)\right)}\x^w    \\
		              & =\sum_{\substack{w\in \N_r^n\\T\subseteq\fDes(w)}}z^{\#T}q^{\fmaj(w,T)}\x^w\\
                &=\sum_{k=0}^nz^{n-k}\sum_{\rho\in \SOmP^r_{n,k}}q^{\fmaj(\rho)}\x^\rho,
	\end{align*}
    where $\fmaj(w,T)$ denotes $\fmaj(\sigma)$ for the element $\sigma\in\SOmP^r_{n,n-\#T}$ whose reading word is $(w,T)$ (Definitions~\ref{def:reading-word-SOmP} and~\ref{def:fmaj-SOmP}), and the last equality is the bijection of Definition~\ref{def:reading-word-SOmP}.
\end{proof}

Next we show that $\fmaj$ and $\sinv$ are equidistributed over $\SOmP^r_{n,k}$. 
By setting $r=1$ and $k=n$ and extracting the coefficient of $x^{(0)}_1x^{(0)}_2\cdots x^{(0)}_n$, this recovers MacMahon's theorem \cite{macmahon1915combinatory}, that the major index and inversion number are equidistributed over $\mathfrak{S}_n$.
We reduce to the case $r=1$ of Wilson \cite{WilsonOMP}, by leveraging a technique of \cite[Lemma~3.6]{ChanRhoades}.
\begin{lemma}\label{lem: fmaj-to-sinv}
	We have
	\begin{equation}
\sum_{\rho\in\SOmP^r_{n,k}}q^{\fmaj(\rho)}\x^\rho=\sum_{\rho\in\SOmP^r_{n,k}}q^{\sinv(\rho)}\x^\rho.
	\end{equation}
\end{lemma} 
\begin{proof}
We first treat $\OmP^r_{n,k}$, that is, the elements with empty zero block, using Wilson's equidistribution theorem for ordered multiset partitions. 
Let $\mathcal A=\{a_1<a_2<\cdots\}$ be a totally ordered alphabet, with variables $\mathbf z=(z_1,z_2,\ldots)$. 
Let $\OmP_{n,k}(\mathcal A)$ be the set of sequences $S_1\mid\cdots\mid S_k$ of $k$ sets with elements in $\mathcal A$ so that $|S_1| + \cdots + |S_k| =n$, where a letter may repeat across different blocks $S_i, S_j$ but not within the same block $S_i$. 
For $\rho\in \OmP_{n,k}(\mathcal A)$, let $(w,T)$ be its descent-starred word, with $w=w_1\cdots w_n$ obtained by sorting each block in decreasing order and dropping the bars, and $T$ the set of indices $i\in [n-1]$ such that $w_i$ and $w_{i+1}$ lie in the same block. 
Let $\Des(w)$ be the set of descents of $w$, that is, indices $i\in [n-1]$ such that $w_i>w_{i+1}$. 
Define statistics
	\begin{equation}
		\maj^1(\rho):=\maj(w,T)=\sum_{i\in \Des(w)}i- \sum_{i \in T}
		\#\Des(w) \cap [i,n-1],
	\end{equation}
    and
	\begin{equation}
		\inv^1(\rho):=\sum_{1\le h<i\le k}\#\left\{s\in S_h:s>\min S_i\right\}.
	\end{equation}
	Wilson's theorem \cite{WilsonOMP} implies that
	\begin{equation}\label{eq:Wilson-used}
		\sum_{\rho\in\OmP_{n,k}(\mathcal A)}
		t^{\maj^1(\rho)}\mathbf z^\rho =
		\sum_{\rho\in\OmP_{n,k}(\mathcal A)}
		t^{\inv^1(\rho)}\mathbf z^\rho,
	\end{equation}
	and that this generating function is symmetric in $\mathbf{z}$.

	Fix $N \geq 1$, and consider only the $rN$ colored letters with bases at most $N$, which in the $<_r$-order are
	\begin{equation}
		1^{r-1}<_r\cdots<_rN^{r-1}<_r1^{r-2}<_r\cdots<_rN^{r-2}<_r\cdots<_r1^0<_r\cdots<_rN^0.
	\end{equation}
	Identify them with $a_1<\cdots<a_{rN}$ in this order, set $z_m=0$ for $m > rN$ and $t=q^r$, and replace the variable $z_{j^i}$ corresponding to the colored letter $j^i$ by $q^i x_j^{(i)}$.
	Then $t^{\maj^1(\rho)}\mathbf z^\rho$ becomes $q^{r\maj^1(\rho)+\eps(\rho)}\x^\rho=q^{\fmaj(\rho)-(r-1)(n-k)}\x^\rho$ by Definition~\ref{def:fmaj-OmP}, since each of the $n-k$ starred positions contributes $r-1$ to $\fmaj(\rho)$. 
    Thus the left-hand side of equation~\eqref{eq:Wilson-used} becomes 
	\begin{equation}
		q^{-(r-1)(n-k)}\sum_{\rho}q^{\fmaj(\rho)}\x^\rho,
	\end{equation} 
    summed over the elements $\rho$ of $\OmP^r_{n,k}$ with bases at most $N$.

	On the right-hand side, we instead identify the letters with the alphabet $a_1<\cdots<a_{rN}$ in the order
	\begin{equation}
		1^{r-1}\prec1^{r-2}\prec\cdots\prec1^0 \prec
		2^{r-1}\prec2^{r-2}\prec\cdots\prec2^0 \prec\cdots
	\end{equation}
    before making the same substitution, which is permitted since the generating function is symmetric in $z_1,\ldots,z_{rN}$.
	Then $\inv^1(\rho)$ becomes $|\sInv(\rho)|$, the number of strong inversion pairs computed with respect
	to $\prec$, and $r|\sInv(\rho)|+\eps(\rho)=\sinv(\rho)-(r-1)(n-k)$ by Definition~\ref{def:sinv-winv}, so the right-hand side becomes
	\begin{equation}
		q^{-(r-1)(n-k)}\sum_{\rho}q^{\sinv(\rho)}\x^\rho
	\end{equation}
    over the same set.
    As $N$ is arbitrary, for all $r \geq 1$ and $n, k\ge 0$,
	\begin{equation}\label{eq:fmaj-sinv-no-zero}
		\sum_{\rho\in\OmP^r_{n,k}}
		q^{\fmaj(\rho)}\x^\rho =
		\sum_{\rho\in\OmP^r_{n,k}}
		q^{\sinv(\rho)}\x^\rho.
	\end{equation}

    It only remains to add the zero block. 
    For $\sigma = (Z,\rho) \in \SOmP^r_{n,k}$ we have $q^{\fmaj(\sigma)}\x^\sigma=q^{\eps'(Z)}\x^{Z} q^{\fmaj(\rho)}\x^\rho$ by Definition~\ref{def:fmaj-SOmP-primed} (which agrees with Definition~\ref{def:fmaj-SOmP} by Lemma~\ref{lem:fmaj-equiv}), and likewise $q^{\sinv(\sigma)}\x^\sigma=q^{\eps'(Z)}\x^{Z}q^{\sinv(\rho)}\x^\rho$ by Definition~\ref{def:sinv-winv}.
    For $1 \leq i \leq r-1$, let $m_i$ be the number of elements of color $i$ in $Z$; summing $q^{\eps'(Z)}\x^Z$ over all the zero blocks with these multiplicities gives
	\begin{equation}
		\prod_{i=1}^{r-1} e_{m_i}\!\left(q^{i-1}\x^{(i)}\right).
	\end{equation}
	Hence applying equation~\eqref{eq:fmaj-sinv-no-zero} gives
	\begin{align*}
		\sum_{\rho\in\SOmP^r_{n,k}}q^{\sinv(\rho)}\x^\rho & =\sum_{0\le m_1,\ldots, m_{r-1}}\prod_{i=1}^{r-1} e_{m_i}\!\left(q^{i-1}\x^{(i)}\right)\times\sum_{\sigma\in \OmP^r_{n-m_1-\cdots-m_{r-1},k}}q^{\sinv(\sigma)}\x^\sigma \\
		& =\sum_{0\le m_1,\ldots, m_{r-1}}\prod_{i=1}^{r-1} e_{m_i}\!\left(q^{i-1}\x^{(i)}\right)\times\sum_{\sigma\in \OmP^r_{n-m_1-\cdots-m_{r-1},k}}q^{\fmaj(\sigma)}\x^\sigma \\
		& =\sum_{\rho\in \SOmP^r_{n,k}}q^{\fmaj(\rho)}\x^\rho.
	\end{align*}
\end{proof}

\subsection{Applying the \texorpdfstring{$\omega$}{omega} involution}\label{subsec:involution}

Next we apply the $\omega$ involution to
\begin{equation}
	D_{n,k}(\x,q) = \sum_{\rho\in \SOmP^r_{n,k}} q^{\sinv(\rho)}\x^\rho,
\end{equation}
using $r$-colored quasisymmetric functions.
We first extend $\omega$ to $\QSym^r_n$ via an involution on signed sets.
\begin{definition}
    The \emph{dual} of a signed set $\sigma=(T,\varepsilon)$ is the signed set $\omega(\sigma):=(T',\varepsilon)$, where for $i \in [n-1]$,
	\begin{itemize}
		\item if $\varepsilon(i)=\varepsilon(i+1)$, then $i\in T'$ if and only if $i\not\in T$;
		\item if $\varepsilon(i)\ne \varepsilon(i+1)$, then $i\in T'$ if and only if $i\in T$.
	\end{itemize}
\end{definition}
It is easy to verify that $\omega(\sigma)$ is a signed set (Definition~\ref{def: signed-set}) and that $\omega$ is an involution of signed sets.
When $(T,\varepsilon)=\sDes(\QQ)$ for some $\QQ\in \SYT^r_n$, the descents in $T$ are of two kinds: those $i$ with $i+1$ in a lower row of the same tableau as $i$, and those $i$ with $i+1$ in a tableau strictly to the left. 
Then $\omega$ swaps descents and non-descents of the first kind, while preserving those of the second kind.
Linearly extending the map $F_{\sigma}(\x)\mapsto F_{\omega(\sigma)}(\x)$ gives an involution of $\QSym^r_n$, which restricts to the usual $\omega$ involution on $\Lambda^r(\x)$, justifying the notation.

\begin{lemma}\label{lem: QSym-involution}
	Suppose $F(\x)\in \Lambda^r(\x)_n$ expands as a $\C$-linear combination of the $F_{\sigma}(\x)$ as
	\begin{equation}
		F(\x)=\sum_{\sigma\in S^r_n}c_\sigma\cdot F_{\sigma}(\x).
	\end{equation}
	Then
	\begin{equation}
		\omega\cdot F(\x)=\sum_{\sigma\in S^r_n}c_\sigma\cdot F_{\omega(\sigma)}(\x).
	\end{equation}
\end{lemma}
\begin{proof}
	It suffices to check this on the basis $\{s_{\llambda}(\x):\llambda\vdash_r n\}$.
	By Lemma~\ref{lem: schur-to-Qsym},
	\begin{equation}
		\omega\cdot s_{\llambda}(\x)=s_{\llambda'}(\x)=\sum_{\QQ\in \SYT^r(\llambda')}
		F_{\sDes(\QQ)}(\x)=\sum_{\QQ\in \SYT^r(\llambda)}F_{\sDes(\QQ')}(\x),
	\end{equation}
	where $\QQ\mapsto \QQ'$ conjugates each tableau.
    Since conjugation swaps descents and non-descents within a tableau, and preserves those across different tableaux, $\sDes(\QQ')=\omega(\sDes(\QQ))$, and the lemma follows.
\end{proof}

\begin{lemma}\label{lem: omega-sinv-to-winv}
	We have
	\begin{equation}
		\omega\left(\sum_{\rho\in\SOmP^r_{n,k}}q^{\sinv(\rho)}\x^\rho\right)=\sum_{\rho\in\SOMP^r_{n,k}}q^{\winv(\rho)}\x^\rho.
	\end{equation}
\end{lemma}
\begin{proof}
    Throughout the proof, write the entries of an element of $\SOmP^r_{n,k}$ in the following order: the zero block, then $S_1,\ldots,S_k$, each in decreasing $\prec$-order; \emph{left} and \emph{right} refer to this order.
	Given $\rho\in \SOmP^r_{n,k}$, we can standardize it by replacing its entries by $1,2,\ldots$ in $\prec$-order, preserving the colors, and breaking ties left-to-right.
	This yields an element $\stdLR(\rho)$ of $\SOP^r_{n,k}$, whose blocks still have their elements in decreasing $\prec$-order, and $\sinv(\stdLR(\rho))=\sinv(\rho)$. 

    Conversely, the $\rho\in \SOmP^r_{n,k}$ with $\stdLR(\rho)=\rho'$ are obtained from $\rho'\in \SOP^r_{n,k}$ by replacing $1,2,\ldots$ by a weakly increasing sequence of positive integers, preserving colors, where the replacements of two entries $i^{k}$ and $(i+1)^{\ell}$ of $\rho'$ may coincide if and only if  either $k>\ell$, or $k=\ell$ and $i+1$ occurs to the right of $i$ (in the latter case, $i^k$ and $(i+1)^k$ lie in different blocks, so no repetition within a block arises).

	Accordingly, for $\rho'\in \SOP^r_{n,k}$ define $\sDes(\rho') = (T,\varepsilon)$ by $\varepsilon(i)=j$ for $i\in [n]$ whenever $i^j$ appears in $\rho'$, and $i \in T$ for $i\in [n-1]$ if and only if either
		      \begin{enumerate}
			      \item $\varepsilon(i)<\varepsilon(i+1)$, or
			      \item $\varepsilon(i)=\varepsilon(i+1)$ and $i+1$ occurs to the left of $i$.
		      \end{enumerate}
	This is a signed set, and $T$ is exactly the set of $i$ at which the replacements may not coincide. 
    Thus the monomials $\x^{\rho}$ with $\stdLR(\rho)=\rho'$ are precisely the terms of $F_{\sDes(\rho')}(\x)$, and
	\begin{equation}\label{eq: sinv-to-QSym}
		\sum_{\rho\in\SOmP^r_{n,k}}q^{\sinv(\rho)}\x^\rho=\sum_{\rho\in \SOP^r_{n,k}}q^{\sinv(\rho)}
		F_{\sDes(\rho)}(\x)=\sum_{\rho\in \SOP^r_{n,k}}q^{\winv(\rho)}F_{\sDes(\rho)}(\x),
	\end{equation}
	where the last equality holds because $\winv$ and $\sinv$ agree on $\SOP^r_{n,k}$.

	Similarly, given $\rho\in \SOMP^r_{n,k}$, we describe a different standardization: replace its entries by $1,2,\ldots$ in $\prec$-order, preserving colors but breaking ties right-to-left instead. 
    This yields $\stdRL(\rho)\in \SOP^r_{n,k}$, with $\winv(\stdRL(\rho))=\winv(\rho)$.

    The $\rho\in\SOMP^r_{n,k}$ with $\stdRL(\rho)=\rho'$ are obtained by the same replacement, where now the replacements of $i^k$ and $(i+1)^\ell$ may coincide if and only if either $k>\ell$, or $k=\ell$ and $i+1$ appears to the \emph{left} of $i$. 
    The set of $i$ at which they may not coincide is $T'$, where $(T',\varepsilon)=\omega(\sDes(\rho'))$, so
	\begin{equation}\label{eq: winv-to-QSym}
	\sum_{\rho\in\SOMP^r_{n,k}}q^{\winv(\rho)}\x^\rho=\sum_{\rho\in \SOP^r_{n,k}}q^{\winv(\rho)}
		F_{\omega(\sDes(\rho))}(\x).
	\end{equation}
    The left-hand side of equation~\eqref{eq: sinv-to-QSym} equals $D_{n,k}\in\Lambda^r(\x)$ by Lemmas~\ref{lem: Dnk-to-fmaj} and~\ref{lem: fmaj-to-sinv}.
    So applying Lemma~\ref{lem: QSym-involution} to it and comparing with equation~\eqref{eq: winv-to-QSym} proves the lemma.
\end{proof}

For example, $\rho=1^11^2\mid 3^2\mid 4^12^1\mid 5^02^0\mid 2^1\in \SOmP^3_{8,4}$ has $\stdLR(\rho)=2^11^2\mid 6^2 \mid7^13^1 \mid 8^05^0 \mid4^1$. 
The $2^1$'s become $3^1$ and $4^1$ left-to-right, while $2^0$ becomes $5^0$ because $2^1\prec 2^0$. 
Similarly $\sigma=1^11^11^2\mid 3^2\mid 4^12^1\mid 5^02^0\mid 2^1\in \SOMP^3_{9,4}$ has $\stdRL(\sigma)=3^12^11^2\mid 7^2\mid 8^15^1\mid 9^06^0\mid 4^1$.

\subsection{A recursion for \texorpdfstring{$I_{n,k}$}{I(n,k)}}\label{subsec:recursion-ideal}
Write
\begin{equation}
	I_{n,k}(\x,q):=\sum_{\rho\in \SOMP^r_{n,k}}q^{\winv(\rho)}\x^\rho,
\end{equation}
so that Lemmas \ref{lem: Dnk-to-fmaj}, \ref{lem: fmaj-to-sinv}, and \ref{lem: omega-sinv-to-winv} give
\begin{equation}\label{eq: Dnk-to-Ink}
	\omega \cdot D_{n,k}=I_{n,k}.
\end{equation}
Since $\omega$ exchanges $e_j(\x^{(i)})^\perp$ and $h_j(\x^{(i)})^\perp$, the recursion of Proposition~\ref{lem: Cnk-recursion} for $D_{n,k}$ is equivalent to the following recursion for $I_{n,k}$.
We say $\rho\in \SOMP^r_{n,k}$ has \emph{content} $\bbeta$ if $\bbeta$ is an $r$-composition such that $\beta^{(i)}_j$ equals the number of times $j^i$ appears in $\rho$.
For an $r$-composition $\bbeta$ we write $\x^\bbeta$ for the monomial in which the exponent of $x^{(c)}_b$ equals $\beta^{(c)}_b$.

\begin{lemma}\label{lem: Ink-recursion}
	For $0\le k\le n$, $0\le i\le r-1$, and $1\le j\le n$,
	\begin{equation}
		h_j\bigl(\x^{(i)}\bigr)^\perp I_{n,k}
		=\sum_{\ell=0}^{\min(j,k)}
		\kappa^{(i)}_{k;j,\ell}\,I_{n-j,k-\ell}
	\end{equation}
	with the coefficients $\kappa^{(i)}_{k;j,\ell}$ of equation~\eqref{eq:common-branching-coefficient}.
\end{lemma}
\begin{proof}
	By equation~\eqref{eq: Dnk-to-Ink}, $I_{n,k}$ is in $\Lambda^r(\x)$. 
    Fix an $r$-composition $\bbeta$ of $n-j$ and choose a positive integer $N$ larger than $\ell(\beta^{(c)})$ for every $0\le c\le r-1$.
	Let $\widehat{\bbeta}$ be the $r$-composition obtained from $\bbeta$ by adding a part $j$ in the $N$th position of $\beta^{(i)}$ (adding zeros in between as needed).
    Then $h_{\widehat{\bbeta}}=h_j(\x^{(i)})h_{\bbeta}$, so
	\begin{equation}\label{eq: xbeta-coeff}
		\left\langle h_j\bigl(\x^{(i)}\bigr)^\perp I_{n,k}, h_{\bbeta} \right\rangle = \left\langle I_{n,k},h_{\widehat{\bbeta}}\right\rangle.
	\end{equation}
	Since $\{h_\lambda\}$ and $\{m_\lambda\}$ are dual bases of the ring of symmetric functions in one set of variables $\Lambda$, the right-hand side is the coefficient of $\x^{\widehat{\bbeta}}$ in $I_{n,k}$, which is the generating function $\sum q^{\winv(\sigma)}$ summed over objects $\sigma \in\SOMP^r_{n,k}$ of content $\widehat{\bbeta}$.

	Deleting the $j$ copies of $N^i$ from $\sigma$, together with any blocks that become empty, say $\ell$ become empty, results in $\rho\in\SOMP^r_{n-j,k-\ell}$ of content $\bbeta$.

	Conversely, $\sigma$ is recovered from $\rho$ by inserting $\ell$ new singleton blocks $\{N^i\}$ and then distributing the remaining $j-\ell$ copies of $N^i$ among the $k$ non-zero blocks of $\sigma$ (the $\ell$ new blocks together with the $k-\ell$ blocks of $\rho$), and into the zero block when $i > 0$.
	We compute the generating function of these choices.

    Since $N^i$ is larger in $\prec$-order than every letter of $\rho$, a copy of $N^i$ forms a weak inversion with each block to its right and with nothing else, and it does not change the minimum of its block.
    A new singleton block with $p$ blocks to its right changes $\winv$ by $i + rp$, so the generating function for the positions of the $\ell$ new blocks among the $k$ blocks of $\sigma$ is $e_\ell \left( q^i,q^{i+r},\ldots,q^{i+r(k-1)} \right) = e_\ell(U_{k,i})$.
	It remains to place the other $j-\ell$ copies of $N^i$.

	An additional copy placed in a block with $p$ blocks to its right changes $\winv$ by $i+(r-1)+rp$, and when $i>0$, a copy placed in the zero block changes it by $i-1$. 
    These weights are the elements of $V_{k,i}$, and since repetitions are allowed, their generating function is $h_{j-\ell}(V_{k,i})$.
	Consequently, for fixed $\ell$, the extensions of $\rho$ to $\sigma$ have generating function $e_\ell(U_{k,i})h_{j-\ell}(V_{k,i}) = \kappa^{(i)}_{k;j,\ell}$, which is independent of $\rho$.
	Summing first over all $\rho\in\SOMP^r_{n-j,k-\ell}$ of content $\bbeta$, and then over $\ell$, gives
	\begin{equation}
		\left\langle h_j\bigl(\x^{(i)}\bigr)^\perp I_{n,k}, h_{\bbeta} \right\rangle = \sum_{\ell=0}^{\min(j,k)} \kappa^{(i)}_{k;j,\ell} \left\langle I_{n-j,k-\ell}, h_{\bbeta} \right\rangle.
	\end{equation}
	Since the $h_{\bbeta}$ span $\Lambda^r(\x)_{n-j}$, the lemma follows.
\end{proof}

Now we are ready to prove Theorem~\ref{thm:main-wreath}, that is, equation~\eqref{eq: wreath-poduct-Frob-formula}.
\begin{proof}[Proof of Theorem~\ref{thm:main-wreath}]
	It suffices to show $C_{n,k}(\x,q)=D_{n,k}(\x,q)$ for all $0\le k\le n$, which we show by strong induction on $n$. 
    For the base case of $n=0$, we have $C_{0,0}=D_{0,0}=1$.
    Assume the claim holds for all $m < n$.
    Since $\omega$ is an isometric involution of $\Lambda^r(\x)$ with $\omega h_j(\x^{(i)})=e_j(\x^{(i)})$, it follows that $e_j(\x^{(i)})^\perp=\omega h_j(\x^{(i)})^\perp\omega$.
    Thus by equation~\eqref{eq: Dnk-to-Ink}, Lemma~\ref{lem: Ink-recursion}, the inductive hypothesis, and Proposition~\ref{lem: Cnk-recursion},
	\begin{equation}\label{eq:kappa-D}
    \begin{aligned}
		e_j\left(\x^{(i)}\right)^\perp D_{n,k} &= \omega h_j\left(\x^{(i)}\right)^\perp I_{n,k}\\
        &= \sum_{\ell=0}^{\min(j,k)}\kappa^{(i)}_{k;j,\ell}\omega I_{n-j,k-\ell} =\sum_{\ell=0}^{\min(j,k)}\kappa^{(i)}_{k;j,\ell} D_{n-j,k-\ell}\\
        &=\sum_{\ell=0}^{\min(j,k)}\kappa^{(i)}_{k;j,\ell} C_{n-j,k-\ell}
        = e_j\left(\x^{(i)}\right)^\perp C_{n,k}
	\end{aligned}
    \end{equation}
	 for all $0\le i\le r-1$, $1 \leq j \leq n$, and $0\le k\le n$.
     For $j>n$, both $e_j(\x^{(i)})^\perp C_{n,k}$ and $e_j(\x^{(i)})^\perp D_{n,k}$ vanish for degree reasons.
By Lemma~\ref{lem: perp-recursion-equivalence}, $D_{n,k} = C_{n,k}$, completing the induction.
\end{proof}

\section{Hook components and consequences}\label{sec:specialization-B}

\subsection{The main theorem at \texorpdfstring{$r=2$}{r=2} and hook components}\label{subsec:B-hook-components}
The main theorem, specialized at $r=2$, yields the following.
\begin{theorem}\label{thm:main-B}
	For $n \geq 1$, the bigraded Frobenius series of $R^{(1,1)}_{\mathfrak{B}_n}$ is given by
	\begin{equation}
		\grFrob(R^{(1,1)}_{\mathfrak{B}_n};q;z)=\sum_{\QQ\in \SYT^2_n}q^{\fmaj(\QQ)}\prod_{j=1}^{\fdes(\QQ)}\left(1+\frac{z}{q^{2j-1}}\right)s_{\shape(\QQ)}(\x).
	\end{equation}
\end{theorem}

From this, we can extract the coefficient of the Schur functions indexed by the shapes $((d),(1^{n-d}))$, which can be
thought of as the type $B$ analogue of the hook shapes. At $d=0$, this recovers the multiplicity of the determinant
character from \cite{SwansonWallach1}.

\begin{proposition}\label{prop:type-B-hooks-11}
	For $n \geq 1$, and $0 \leq d \leq n$, we have
	\begin{equation}
		\left\langle \grFrob(R^{(1,1)}_{\mathfrak{B}_n};q;z), s_{((d),(1^{n-d}))}(\x) \right\rangle = \qbinom{n}{d}_{q^2} \prod_{j=1}^{n-d} (z+q^{2j-1}).
	\end{equation}
\end{proposition}

\begin{proof}
	By Theorem~\ref{thm:main-B}, we write
	\begin{equation}\begin{aligned}
		 & \left\langle \grFrob(R^{(1,1)}_{\mathfrak{B}_n};q;z), s_{((d),(1^{n-d}))}(\x) \right\rangle                                                                             \\
		 & = \left\langle \sum_{\QQ\in \SYT^2_n}q^{\fmaj(\QQ)}\prod_{j=1}^{\fdes(\QQ)}\left(1+\frac{z}{q^{2j-1}}\right)s_{\shape(\QQ)}(\x),  s_{((d),(1^{n-d}))}(\x) \right\rangle \\
		 & = \sum_{\QQ\in \SYT^2_n((d),(1^{n-d}))}q^{\fmaj(\QQ)}\prod_{j=1}^{\fdes(\QQ)}\left(1+\frac{z}{q^{2j-1}}\right).
	\end{aligned}\end{equation}
	Now for any $D \subseteq \{0,\ldots,n-1\}$ with $|D| = n-d$, there exists a unique standard Young bitableau ($2$-tableau) of shape $((d),(1^{n-d}))$ with flag descent set $D$.
	Let $||D||$ denote the sum of the elements of $D$; then $\fmaj(\QQ) = 2||D|| + n-d$ and $\fdes(\QQ) = n-d$.
	Hence we continue, writing that the previous equation is equal to
	\begin{equation}\label{eq:to-factor}
		\sum_{\substack{D \subseteq \{0,\ldots,n-1\}\\|D|=n-d}}q^{2||D||+n-d}\prod_{j=1}^{n-d}\left(1+\frac{z}{q^{2j-1}}\right).
	\end{equation}
	One can check that
	\begin{equation}
		\sum_{\substack{D \subseteq \{0,\ldots,n-1\}\\|D|=n-d}}q^{2||D||+n-d} = q^{(n-d)^2}\qbinom{n}{d}_{q^2}.
	\end{equation}
	Since $\prod_{j=1}^{n-d}\left(1+\frac{z}{q^{2j-1}}\right)$ is not a function of $D$, it can be factored out of equation~\eqref{eq:to-factor}, yielding
	\begin{equation}
		\left(\prod_{j=1}^{n-d}\left(1+\frac{z}{q^{2j-1}}\right)\right) q^{(n-d)^2}\qbinom{n}{d}_{q^2} = \qbinom{n}{d}_{q^2} \prod_{j=1}^{n-d}\left(z+q^{2j-1}\right),
	\end{equation}
	as desired.
\end{proof}

We recall the following.
\begin{proposition}[{\cite[Proposition 3.3]{JiangLentfer}}]
	\label{prop:one-set}
	For $n \geq 1$,
	\begin{equation}
		\grFrob(R_{\mathfrak{B}_n}^{(0,1)}; w) = \sum_{i=0}^{n} w^i s_{((n-i),(1^i))}(\x).
	\end{equation}
\end{proposition}

It is easy to establish the following surjection.
\begin{lemma}\label{lem:tensor-surjection}
Fix a finite group $G \subseteq \GL_n$. 
There is a surjection of trigraded $G$-modules 
\begin{equation} 
R^{(1,1)}_G \otimes R^{(0,1)}_G \twoheadrightarrow R^{(1,2)}_G. 
\end{equation}
\end{lemma}
\begin{proof} The left-hand side is the quotient of $\C[\bm{x},\bm{\theta}^{(1)},\bm{\theta}^{(2)}]$ by the ideal generated by $I^{(1,1)}_G$, in the variables $\bm{x},\bm{\theta}^{(1)}$, and $I^{(0,1)}_G$, in the variables $\bm{\theta}^{(2)}$. Both are generated by diagonal $G$-invariants with no constant term, so this ideal is contained in $I^{(1,2)}_G$. \end{proof}

We will establish an upper bound on the trigraded multiplicity of the determinant character of $R_{\mathfrak{B}_n}^{(1,2)}$,
by computing the multiplicity of the determinant character in $R_{\mathfrak{B}_n}^{(1,1)} \otimes
	R_{\mathfrak{B}_n}^{(0,1)}$. The method is similar to that first used in \cite{HaglundSergel}, and subsequently in
\cite{lentfer2026signcharactertriagonalfermionic, GonzalezLentferMularczyk} (all in type $A$) and in
\cite{JiangLentfer} (in type $B$).
\begin{proposition}\label{prop:B(1,1)(0,1)}
	For $n \geq 1$, we have
	\begin{equation}
		\left\langle \grFrob(R_{\mathfrak{B}_n}^{(1,1)} \otimes R_{\mathfrak{B}_n}^{(0,1)};q;z, w), s_{(\varnothing,(1^{n}))}(\x) \right\rangle = \sum_{d=0}^n w^d \qbinom{n}{d}_{q^2} \prod_{j=1}^{n-d} (z+q^{2j-1}).
	\end{equation}
	In particular, at $q=z=w=1$, this evaluates to $3^n$.
\end{proposition}

\begin{proof}
	By \cite[Lemma 5.3]{JiangLentfer} and \cite[Lemma 5.4]{JiangLentfer}, we have
	\begin{equation}
		\begin{aligned}
			\langle & \grFrob(R_{\mathfrak{B}_n}^{(1,1)} \otimes R_{\mathfrak{B}_n}^{(0,1)}; q;z,w), s_{(\varnothing,(1^n))}(\x)\rangle                                                                                                                                \\
			        & = \sum_{(\lambda^{(0)},\lambda^{(1)}) \vdash_2 n} \langle \grFrob(R_{\mathfrak{B}_n}^{(1,1)}; q;z), s_{(\lambda^{(0)},\lambda^{(1)})}(\x) \rangle \langle \grFrob(R_{\mathfrak{B}_n}^{(0,1)}; w), s_{(\lambda^{(1)'},\lambda^{(0)'})}(\x)\rangle \\
			        & = \sum_{d=0}^n w^d \langle \grFrob(R_{\mathfrak{B}_n}^{(1,1)}; q;z), s_{((d),(1^{n-d}))} (\x)\rangle                                                                                                                                             \\
			        & = \sum_{d=0}^n w^d\qbinom{n}{d}_{q^2} \prod_{j=1}^{n-d} (z+q^{2j-1}),
		\end{aligned}
	\end{equation}
	where we applied Proposition~\ref{prop:type-B-hooks-11} and Proposition~\ref{prop:one-set}.
	Finally, the specialization at $q=z=w=1$ follows easily by using the binomial theorem.
\end{proof}

For two polynomials $P, Q \in \Z_{\geq 0}[q,z,w]$, let $P \leq Q$ mean that $Q - P \in \Z_{\geq 0}[q,z,w]$. Then for
symmetric functions, let $f \leq g$ mean that $[s_{\llambda}]f \leq [s_{\llambda}]g$ for all $\llambda$.

\begin{corollary}\label{cor:B12}
	For $n \geq 1$, we have
	\begin{equation}
		\left\langle \grFrob(R^{(1,2)}_{\mathfrak{B}_n};q;z, w), s_{(\varnothing,(1^{n}))}(\x) \right\rangle \leq \sum_{d=0}^n w^d \qbinom{n}{d}_{q^2} \prod_{j=1}^{n-d} (z+q^{2j-1}).
	\end{equation}
	In particular, at $q=z=w=1$, the right-hand side is $3^n$.
\end{corollary}

\begin{proof}
	The surjection of Lemma~\ref{lem:tensor-surjection} gives
	\begin{equation}
		\grFrob(R_{\mathfrak{B}_n}^{(1,2)}; q;z,w) \leq \grFrob(R_{\mathfrak{B}_n}^{(1,1)} \otimes R_{\mathfrak{B}_n}^{(0,1)}; q;z,w).
	\end{equation}
    Extracting the coefficient of $s_{(\varnothing,(1^n))}(\x)$ and applying Proposition~\ref{prop:B(1,1)(0,1)} gives the result.
\end{proof}

\subsection{Hook components in type \texorpdfstring{$A$}{A}}\label{subsec:A-hook-components}
The same proof methods allow us to obtain the analogous results in type $A$.

\begin{proposition}\label{prop:type-A-hooks-11}
	For $n \geq 1$, and $0 \leq d \leq n-1$, we have
	\begin{align}
		\left\langle \grFrob(R^{(1,1)}_{\mathfrak{S}_n};q;z), s_{(d+1, 1^{n-d-1})}(\x) \right\rangle & = \qbinom{n-1}{d}_q \prod_{j=1}^{n-1-d} (z+q^j)                                      \\
		& =  \sum_{i=0}^{n-1} z^{i} q^{\binom{n-i-d}{2}}\qbinom{n-1-d}{i}_q \qbinom{n-1}{d}_q.
	\end{align}
\end{proposition}

\begin{proof}
	By equation~\eqref{eq:type-A-product-form}, we write
	\begin{equation}\begin{aligned}
		 & \left\langle \grFrob(R^{(1,1)}_{\mathfrak{S}_n};q;z), s_{(d+1,1^{n-d-1})}(\x) \right\rangle                                                                     \\
		 & = \left\langle \sum_{T \in \SYT_n} q^{\maj(T)} \prod_{j=1}^{\des(T)} \left(1+\frac{z}{q^{j}}\right)\, s_{\shape(T)}(\x),  s_{(d+1,1^{n-d-1})}(\x) \right\rangle \\
		 & = \sum_{T\in \SYT_n(d+1,1^{n-d-1})}q^{\maj(T)}\prod_{j=1}^{\des(T)}\left(1+\frac{z}{q^{j}}\right).
	\end{aligned}\end{equation}
	Now for any $D \subseteq \{1,\ldots,n-1\}$ with $|D| = n-1-d$, there exists a unique standard Young tableau of shape $(d+1,1^{n-d-1})$ with descent set $D$.
	Let $||D||$ denote the sum of the elements of $D$; then $\maj(T) = ||D||$ and $\des(T) = n-1-d$.
	Hence we continue, writing that the previous equation is equal to
	\begin{equation}\label{eq:to-factor-A}
		\sum_{\substack{D \subseteq \{1,\ldots,n-1\}\\|D|=n-1-d}}q^{||D||}\prod_{j=1}^{n-1-d}\left(1+\frac{z}{q^{j}}\right).
	\end{equation}
	One can check that
	\begin{equation}
		\sum_{\substack{D \subseteq \{1,\ldots,n-1\}\\|D|=n-1-d}}q^{||D||} = q^{\binom{n-d}{2}}\qbinom{n-1}{d}_{q}.
	\end{equation}
	Hence equation~\eqref{eq:to-factor-A} can be further simplified to
	\begin{equation}\begin{aligned}
		\left(\prod_{j=1}^{n-1-d}\left(1+\frac{z}{q^{j}}\right)\right) q^{\binom{n-d}{2}}\qbinom{n-1}{d}_{q} & = \left(\prod_{j=1}^{n-1-d}\left(q^j+z\right)\right) \qbinom{n-1}{d}_{q}             \\
		& = \sum_{i=0}^{n-1} z^i q^{\binom{n-i-d}{2}} \qbinom{n-1-d}{i}_q \qbinom{n-1}{d}_{q},
	\end{aligned}\end{equation}
	where the last line follows by the $q$-binomial theorem.
\end{proof}

\begin{remark}
	Using the Theta operators of \cite{DadderioIraciVandenWyngaerd2021} and further results of \cite{IraciNadeauVandenWyngaerd2024}, Proposition~\ref{prop:type-A-hooks-11} can also be deduced from equation~\eqref{eq:type-A-product-form} and the $\ell=0$ case of \cite[Theorem 8.4]{Lentfer2024} (which more generally contains a conjectural formula for the trigraded multiplicity of the hook characters of $R_{\mathfrak{S}_n}^{(1,2)}$).
\end{remark}

We recall the following.
\begin{proposition}[{\cite[Lemma 4.10]{HaglundSergel}}]\label{prop:one-set-A}
	For $n \geq 1$,
	\begin{equation}
		\grFrob(R_{\mathfrak{S}_n}^{(0,1)}; w) = \sum_{i=0}^{n-1} w^i s_{(n-i,1^i)}(\x).
	\end{equation}
\end{proposition}

We establish an upper bound on the trigraded multiplicity of the sign character of $R_{\mathfrak{S}_n}^{(1,2)}$, by
first establishing it for $R_{\mathfrak{S}_n}^{(1,1)} \otimes R_{\mathfrak{S}_n}^{(0,1)}$. The ungraded value of
$3^{n-1}$ was conjectured by F. Bergeron \cite{Bergeron2020}.
\begin{proposition}\label{prop:A(1,1)(0,1)}
	For $n \geq 1$, we have
	\begin{equation}
		\left\langle \grFrob(R^{(1,1)}_{\mathfrak{S}_n} \otimes R_{\mathfrak{S}_n}^{(0,1)};q;z, w), s_{(1^{n})}(\x) \right\rangle = \sum_{i,d=0}^{n-1}z^i w^d q^{\binom{n-i-d}{2}} \qbinom{n-1}{d}_q\qbinom{n-1-d}{i}_q.
	\end{equation}
	In particular, at $q=z=w=1$, this evaluates to $3^{n-1}$.
\end{proposition}

\begin{proof}
	By \cite{Bessenrodt} and the fact that the Kronecker coefficient $g(\lambda,\mu,(1^n))$ is $1$ if $\lambda = \mu'$ and $0$ otherwise, we have
	\begin{equation}
		\begin{aligned}
			\langle & \grFrob(R_{\mathfrak{S}_n}^{(1,1)} \otimes R_{\mathfrak{S}_n}^{(0,1)}; q;z,w), s_{(1^n)}(\x)\rangle                                                                         \\
			        & = \sum_{\lambda \vdash n} \langle \grFrob(R_{\mathfrak{S}_n}^{(1,1)}; q;z), s_{\lambda}(\x) \rangle \langle \grFrob(R_{\mathfrak{S}_n}^{(0,1)}; w), s_{\lambda'}(\x)\rangle \\
			        & = \sum_{d=0}^{n-1} w^d \langle \grFrob(R_{\mathfrak{S}_n}^{(1,1)}; q;z), s_{(d+1,1^{n-d-1})} (\x)\rangle                                                                    \\
			        & = \sum_{d=0}^{n-1} w^d\sum_{i=0}^{n-1} z^{i} q^{\binom{n-i-d}{2}}\qbinom{n-1-d}{i}_q \qbinom{n-1}{d}_q,
		\end{aligned}
	\end{equation}
	where we applied Proposition~\ref{prop:type-A-hooks-11} and Proposition~\ref{prop:one-set-A}.
	Finally, the specialization at $q=z=w=1$ follows easily by using the binomial theorem.
\end{proof}

\begin{corollary}\label{cor:A12}
	For $n \geq 1$, we have
	\begin{equation}
		\left\langle \grFrob(R^{(1,2)}_{\mathfrak{S}_n};q;z, w), s_{(1^{n})}(\x) \right\rangle \leq \sum_{i,d=0}^{n-1}z^i w^d q^{\binom{n-i-d}{2}} \qbinom{n-1}{d}_q\qbinom{n-1-d}{i}_q.
	\end{equation}
	In particular, at $q=z=w=1$, the right-hand side is $3^{n-1}$.
\end{corollary}

\begin{proof}
	The surjection of Lemma~\ref{lem:tensor-surjection} gives
	\begin{equation}
		\grFrob(R_{\mathfrak{S}_n}^{(1,2)}; q;z,w) \leq \grFrob(R_{\mathfrak{S}_n}^{(1,1)} \otimes R_{\mathfrak{S}_n}^{(0,1)}; q;z,w).
	\end{equation}
    Extracting the coefficient of $s_{(1^n)}(\x)$ and applying Proposition~\ref{prop:A(1,1)(0,1)} gives the result.
\end{proof}

\subsection{A proof of a conjecture of Bergeron on hook components}\label{subsec:Bergeron}
In this section, we record a short proof of a conjecture of Bergeron \cite{BergeronOPAC} on hook components of
$R_{\mathfrak{S}_n}^{(1,1)}$, which follows from the existing literature.

Define a \emph{super Schur function} (see \cite{BereleRegev,ChengWangBook}) by
\begin{equation}\label{eq:superschur}
	s_\lambda(\mathbf{q}/\mathbf{z}) = \sum_{\nu \subseteq \lambda} s_\nu(\mathbf{q}) s_{\lambda'/\nu'}(\mathbf{z}).
\end{equation}

We will use the following result, which was originally conjectured by Bergeron {\cite[Conjecture 1]{Bergeron2020}}.
Define $P(k,j,n)$ to be the set of all partitions of length $\ell(\lambda) \leq n$ such that $\lambda_{k+1} \leq j$.
\begin{theorem}[{\cite[Corollary 1.3]{Lentfer-Supersymmetry}}]\label{thm:DSS}
	Fix a positive integer $n$.
	For partitions $\lambda$ with $\ell(\lambda) \leq n$, and $\mu \vdash n$, there exist nonnegative integer coefficients $c_{\lambda,\mu}$ such that for any $(k,j)$, the multigraded Frobenius series of $R_{\mathfrak{S}_n}^{(k,j)}$ is
	\begin{equation}\label{eq:theorem}
		\grFrob(R_{\mathfrak{S}_n}^{(k,j)}; \mathbf{q};\mathbf{z}) = \sum_{\lambda \in P(k,j,n)} \sum_{\mu \vdash n} c_{\lambda,\mu} s_\lambda(\mathbf{q}/\mathbf{z})s_\mu.
	\end{equation}
\end{theorem}

Now at $k=j=1$, equation~\eqref{eq:theorem} becomes
\begin{equation}\label{eq:theorem-11}
	\grFrob(R_{\mathfrak{S}_n}^{(1,1)}; q;z) = \sum_{\lambda \text{ is a hook, } \ell(\lambda) \leq n} \sum_{\mu \vdash n} c_{\lambda,\mu} s_\lambda(q/z)s_\mu.
\end{equation}

Define
\begin{equation}
	\mathbf{c}_\mu(\mathbf{q}/\mathbf{z}):= \langle \grFrob(R_{\mathfrak{S}_n}^{(\infty, \infty)};\mathbf{q};\mathbf{z}), s_\mu \rangle = \sum_{\lambda}c_{\lambda,\mu} s_{\lambda}(\mathbf{q}/\mathbf{z}),
\end{equation}
where $R_{\mathfrak{S}_n}^{(\infty, \infty)}$ is the direct limit of the $R_{\mathfrak{S}_n}^{(k, j)}$ \cite{Lentfer-Supersymmetry}.
\begin{remark}
    Bergeron \cite{BergeronOPAC} originally defined the $\mathbf{c}_\mu(\mathbf{q}/\mathbf{z})$ as the multiplicities in the $\mathfrak{S}_n$-module of multivariate diagonal harmonics $H_{\mathfrak{S}_n}^{(\infty,0)}$, which is isomorphic to the coinvariant ring $R_{\mathfrak{S}_n}^{(\infty,0)}$.
    By Theorem~\ref{thm:DSS}, we can interpret them as we do here.
\end{remark}
With one variable in each set $\mathbf{q}$ and $\mathbf{z}$, the specialization $s_\lambda(q/z)$ is $0$ unless $\lambda$ is a hook shape.
Thus
\begin{equation}\label{eq:c-q-z}
	\mathbf{c}_\mu(q/z)= \sum_{\lambda \text{ is a hook}}c_{\lambda,\mu} s_{\lambda}(q/z) = \langle \grFrob(R_{\mathfrak{S}_n}^{(1, 1)};q;z), s_\mu \rangle,
\end{equation}
by equation~\eqref{eq:theorem-11}.

Now we are ready to prove the following conjecture of Bergeron.
\begin{corollary}[Conjecture 4, \cite{BergeronOPAC}]\label{cor:bergeron-conjecture-4}
	Fix $n \geq 1$. For $0 \leq b \leq n-1$,
	\begin{equation}
		\mathbf{c}_{(n-b,1^b)}(q/z) = \qbinom{n-1}{b}_q (q+z)(q^2+z)\cdots (q^b+z).
	\end{equation}
\end{corollary}

\begin{proof}
	By Proposition~\ref{prop:type-A-hooks-11} with $d = n-b-1$, the right-hand side of equation~\eqref{eq:c-q-z} at $\mu = (n-b,1^b)$ is $\qbinom{n-1}{b}_q \prod_{j=1}^{b}(z+q^j)$.
\end{proof}

As a consequence, we determine an explicit formula for universal series coefficients $c_{\lambda,\mu}$, where both
$\lambda$ and $\mu$ are hook shapes. This generalizes \cite[Corollary 4.6]{Lentfer-Supersymmetry} from the symmetric
group sign character to hook characters.
Let $\langle q^k \rangle p(q)$ denote the coefficient of $q^k$ from the polynomial $p(q)$.

\begin{corollary}\label{cor:hook-hook}
	Fix $n \geq 1$. Let $\mu = (n-b,1^b)$ be a hook shape, and let $\lambda$ be any non-empty hook shape. Then:
	\begin{itemize}
		\item for $0 \leq i \leq b-1$ and $k \geq 0$, we have 
        \begin{equation}
			c_{(\binom{b+1-i}{2}+k,1^i),(n-b,1^b)} = \langle q^k \rangle \qbinom{b-1}{i}_q \qbinom{n-1}{b}_q;
		\end{equation}
		\item $c_{\lambda, (n-b,1^b)} = 0$ for every other hook $\lambda$.
	\end{itemize}
\end{corollary}

\begin{proof}
	By Proposition~\ref{prop:type-A-hooks-11} with $d = n-b-1$, we have that
	\begin{align*}
		\langle\grFrob & (R_{\mathfrak{S}_n}^{(1, 1)};q;z), s_{(n-b,1^{b})}\rangle = \qbinom{n-1}{b}_q \sum_{i=0}^{n-1} z^{i} q^{\binom{b+1-i}{2}}\qbinom{b}{i}_q \\
		               & = \qbinom{n-1}{b}_q \sum_{i=0}^{n-1} z^{i} q^{\binom{b+1-i}{2}}\left( \qbinom{b-1}{i}_q + q^{b-i} \qbinom{b-1}{i-1}_q\right)             \\
		               & = \sum_{i=0}^{n-1} z^{i} \sum_{k \geq 0} \Bigg(q^{\binom{b+1-i}{2}+k} \langle q^k \rangle \qbinom{b-1}{i}_q \qbinom{n-1}{b}_q            \\
		               & \qquad + q^{\binom{b+1-i}{2}+b-i+k} \langle q^k \rangle \qbinom{b-1}{i-1}_q \qbinom{n-1}{b}_q\Bigg).
	\end{align*}
	Then, by shifting $i \mapsto i+1$ in the second part,
	\begin{align*}
		 & = \sum_{i=0}^{n-1} \sum_{k \geq 0} \left( z^i q^{\binom{b+1-i}{2}+k}  + z^{i+1}q^{\binom{b+1-i}{2}+k-1} \right)\langle q^k \rangle \qbinom{b-1}{i}_q \qbinom{n-1}{b}_q \\
		 & = \sum_{i=0}^{n-1} \sum_{k \geq 0} s_{(\binom{b+1-i}{2}+k,1^i)}(q/z)\langle q^k \rangle \qbinom{b-1}{i}_q \qbinom{n-1}{b}_q.
	\end{align*}
\end{proof}

\section{Recovering the Hilbert series}\label{sec:combinatorial-consequences}

In this section, we specialize Theorem~\ref{thm:main-wreath} to the bigraded Hilbert series of
$R^{(1,1)}_{\Z_r\wr\mathfrak{S}_n}$ and show that the result agrees with the formula conjectured by Sagan--Swanson
\cite[Conjecture~5.3]{SSreflection} and proven by Bhattacharya--Rhoades \cite[Corollary~3.22]{BR}. Since the proof of
Theorem~\ref{thm:main-wreath} takes the monomial basis of Theorem~\ref{thm: monomial-basis} as input, this does not give
an independent proof of the Hilbert series.
We also obtain generating functions for the statistics $\fmaj$
and $\sinv$ on $\SOP^r_{n,k}$ (Proposition~\ref{prop:Hilb-SOP}).

Let $\llambda = (\lambda^{(0)},\ldots,\lambda^{(r-1)}) \vdash_r n$, and write $n_i:= |\lambda^{(i)}|$. Recall that
\begin{equation}
	s_{\llambda}(\x)= \prod_{i=0}^{r-1}s_{\lambda^{(i)}}(\x^{(i)}),\qquad m_{\llambda}(\x)=\prod_{i=0}^{r-1}m_{\lambda^{(i)}}(\x^{(i)}),
\end{equation}
and
\begin{equation}
	s_{\llambda}(\x)=\sum_{\mmu}
	K_{\llambda,\mmu}m_{\mmu}(\x),
\end{equation}
where ${\mmu} = (\mu^{(0)},\ldots,\mu^{(r-1)})$, $\mu^{(i)}$ is a partition of $n_i$ that is dominated by $\lambda^{(i)}$ for all $i$, and $K_{\llambda,\mmu}:= \prod_{i=0}^{r-1} K_{\lambda^{(i)},\mu^{(i)}}$ is a product of Kostka numbers.
In particular if $\mu^{(i)}=1^{n_i}$ for all $i$,
\begin{equation}\label{eq:Kostka}	K_{\llambda,\mmu}=\prod_{i=0}^{r-1}f^{\lambda^{(i)}}.
\end{equation}

Recall from \eqref{eq:Hilb-from-Frob} that the Hilbert series of a bigraded $\Z_r \wr \mathfrak{S}_n$-module $V$ may be
recovered from its bigraded Frobenius series via
\begin{equation}
	\Hilb(V;q;z) = \left\langle \grFrob(V;q;z), \big(h_1(\x^{(0)}) + \cdots + h_1(\x^{(r-1)})\big)^n \right\rangle,
\end{equation}
and
\begin{equation}
	\big(h_1(\x^{(0)}) + \cdots + h_1(\x^{(r-1)})\big)^n=\sum_{n_0+\cdots +n_{r-1}=n}
	\binom{n}{n_0,\ldots,n_{r-1}}
	\prod_{i=0}^{r-1}h_{1^{n_i}}(\x^{(i)}).
\end{equation}

Since
$\langle m_{\llambda}(\x),h_{\mmu}(\x)
	\rangle=\delta_{\llambda,{\mmu}}$,
the pairing $\langle s_{\llambda}(\x),h_{\mmu}(\x)\rangle$ equals $K_{\llambda,\mmu}$ when $|\mu^{(i)}| = n_i$ for
all $i$, and is zero otherwise. Thus Theorem~\ref{thm:main-wreath} gives
\begin{equation}\label{eq:Hilbert}
	\Hilb\big(R^{(1,1)}_{\Z_r\wr\mathfrak{S}_n};q;z\big) =\sum_{\QQ\in \SYT^r_n}q^{\fmaj(\QQ)}\prod_{j=1}^{\fdes(\QQ)}\left(1+\frac{z}{q^{rj-r+1}}\right)f^{\shape(\QQ)},
\end{equation}
where
\begin{equation}
	f^{\llambda}:=\binom{n}{|\lambda^{(0)}|,\ldots,|\lambda^{(r-1)}|}\prod_{i=0}^{r-1}f^{\lambda^{(i)}}
\end{equation}
is the number of standard $r$-tableaux of shape $\llambda$.

Recall the $r$-colored RSK correspondence $g \mapsto (\PP(g),\QQ(g))$ and its restriction to $\Z_r\wr\mathfrak{S}_n$
(Corollary~\ref{cor:RS-colored-permutations}).
The bijection in equation~\eqref{eq:RS-colored-permutations} extends to $r$-colored super ordered set partitions in $\SOP^r_{n,k}$ as follows.
Identify $\rho \in \SOP^r_{n,k}$ with its reading word $(g,S)$, where $g \in \Z_r\wr\mathfrak{S}_n$, $S\subseteq \fDes(g)$, and $|S|=n-k$ (Definition~\ref{def:reading-word-SOmP}).
Associate to $\rho$ the triple $(\PP(g), \QQ(g), S)$. This gives a bijection between $\SOP^r_{n,k}$ and the set of triples $(\PP,\QQ,S)$ where $\PP$ and $\QQ$ are standard $r$-tableaux of the same shape and size $n$, $S\subseteq \fDes(\QQ)$, and $|S|=n-k$.

Analogously to Definition~\ref{def:fmaj-SOmP}, we define
\begin{equation}
	\fmaj(\QQ,S):=\fmaj(\QQ)-\sum_{i\in S}(r|\fDes(\QQ)\cap [i,n-1]|-(r-1)).
\end{equation}
By equation~\eqref{eq:RSK-fDes-fmaj} and Definition~\ref{def:fmaj-SOmP}, we have that $\fmaj(\QQ,S)=\fmaj(\rho)$ whenever
$\rho$ corresponds to $(\PP,\QQ,S)$. Expanding the product, we have
\begin{equation}
	\sum_{\QQ\in \SYT^r_n}q^{\fmaj(\QQ)}\prod_{j=1}^{\fdes(\QQ)}\left(1+\frac{z}{q^{rj-r+1}}\right)=
	\sum_{\QQ\in \SYT^r_n}\sum_{S\subseteq \fDes(\QQ)}z^{|S|}q^{\fmaj(\QQ,S)},
\end{equation}
where we used that the $j$th largest element $i$ of $\fDes(\QQ)$ satisfies $|\fDes(\QQ)\cap [i,n-1]| = j$.

We can now conclude
\begin{align}
	\Hilb\big(R^{(1,1)}_{\Z_r\wr\mathfrak{S}_n};q;z\big) & =\sum_{\QQ\in \SYT^r_n}q^{\fmaj(\QQ)}\prod_{j=1}^{\fdes(\QQ)}\left(1+\frac{z}{q^{rj-r+1}}\right)f^{\shape(\QQ)}    \notag    \\
	& = \sum_{\substack{\PP,\QQ\in \SYT^r_n\\ \shape(\PP)=\shape(\QQ)}}\ \sum_{S\subseteq \fDes(\QQ)}z^{|S|}q^{\fmaj(\QQ,S)}\notag \\
	& = \sum_{k=0}^n z^{n-k}\sum_{\rho\in \SOP^r_{n,k}}q^{\fmaj(\rho)}.\label{eq:last-expression}
\end{align}
It only remains to identify the last expression with the formula of \cite[Corollary 3.22]{BR}.

For an integer $m$ and a variable $t$, write $[m]_t:= 1 + t + \cdots + t^{m-1}$ if $m \geq 1$, and $[m]_t:= 0$
otherwise. Following \cite{BR}, define $\Stir^{(r)}_q(n,k)$ for integers $n \geq 0$ and $k$ by
$\Stir^{(r)}_q(0,k):= \delta_{k,0}$ and
\begin{equation}\label{eq:Stir-recursion}
	\Stir^{(r)}_q(n,k):= \Stir^{(r)}_q(n-1,k-1) + [kr+r-1]_q  \Stir^{(r)}_q(n-1,k)
\end{equation}
for $n \geq 1$, so that $\Stir^{(r)}_q(n,k) = 0$ unless $0 \leq k \leq n$.

\begin{proposition}\label{prop:Hilb-SOP}
	For $0 \leq k \leq n$, we have
	\begin{equation}\label{eq:Hilb-SOP}
		\sum_{\rho \in \SOP^r_{n,k}} q^{\fmaj(\rho)}
		= \sum_{\rho \in \SOP^r_{n,k}} q^{\sinv(\rho)}
		= [r]_q [2r]_q \cdots [kr]_q \, \Stir^{(r)}_q(n,k).
	\end{equation}
\end{proposition}

\begin{proof}
	The first equality follows from Lemma~\ref{lem: fmaj-to-sinv}, by extracting the coefficients of $\mathbf{x}^\rho$ for those $\rho$ which are $r$-colored super ordered set partitions.

	For the second equality, for $0 \leq k \leq n$, set \begin{equation}
		S(n,k):=\sum_{\rho \in \SOP^r_{n,k}} q^{\sinv(\rho)},
	\end{equation}
	and otherwise set $S(n,k) := 0$.

	We claim that
	\begin{equation}\label{eq:S-recursion}
		S(n,k)=[kr]_q S(n-1,k-1)+[kr+r-1]_q S(n-1,k),
	\end{equation}
	for $n \geq 1$.

	Every $\rho\in \SOP^r_{n,k}$ is obtained in exactly one way, by adding the letter $n^i$, for some color $i$, to the element $\rho^0$ of $\SOP^r_{n-1}$ (obtained from $\rho$ by deleting $n^i$, together with its block if $\{n^i\}$ is a block).
	Since $n^i$ is the largest letter of $\rho$ with respect to the order $\prec$, adding $n^i$ to a non-zero block does not change the minimum with respect to $\prec$ of that block, and $n^i$ forms a strong inversion with every block strictly to the right of its block.
	Thus if $n^i$ lies in a non-zero block of $\rho$, the strong inversions of $\rho$ which are not strong inversions of $\rho^0$ are the pairs $(n^i,S)$ with $S$ a block strictly to the right of the block containing $n^i$. If $n^i \in Z$, then $\rho$ and $\rho^0$ have the same strong inversions.
	We compute the change in each of the four terms of Definition~\ref{def:sinv-winv}; there are two possibilities:
	\begin{itemize}
		\item Case 1: $\rho^0=Z\mid S_1\mid \cdots \mid S_{k-1} \in \SOP^r_{n-1,k-1}$ and $n^i$ is added as a singleton block.
		      If it is placed so that exactly $j$ blocks are to its right, this creates $j$ new strong inversions and adds $i$ to the sum of colors, while $\ell-k$ and $\eps'(Z)$ are unchanged.
		      Thus $\sinv(\rho)=\sinv(\rho^0)+rj+i$.
		      As $j$ varies over $\{0,\ldots, k-1\}$ and $i$ over $\{0,\ldots, r-1\}$, the values $rj+i$ run over $\{0,\ldots, kr-1\}$, each exactly once.
		      Thus these $\rho$ contribute $[kr]_q S(n-1,k-1)$ to $S(n,k)$.

		\item Case 2: $\rho^0=Z\mid S_1\mid \cdots \mid S_{k} \in \SOP^r_{n-1,k}$ and $n^i$ is added to an existing block.
		      If $n^i$ is added to block $S_j$, this creates $k-j$ new strong inversion pairs, adds $i$ to the sum of colors, and increases $\ell-k$ by $1$, so $\sinv(\rho)=\sinv(\rho^0)+r(k-j)+i+(r-1)$.
		      As $j$ varies over $\{1,\ldots, k\}$ and $i$ over $\{0,\ldots, r-1\}$, these values run over $\{r-1,\ldots, kr+r-2\}$, each exactly once.
		      If $n^i$ is added to $Z$, then $i\neq 0$, and only $\eps'(Z)$ changes, so $\sinv(\rho)=\sinv(\rho^0)+(i-1)$ takes each value in $\{0,\ldots, r-2\}$ exactly once.
		      Together, every value in $\{0,\ldots, kr+r-2\}$ occurs exactly once, so these $\rho$ contribute $[kr+r-1]_q S(n-1,k)$ to $S(n,k)$.
	\end{itemize}
	Summing the two cases yields equation~\eqref{eq:S-recursion}.
	Multiplying equation~\eqref{eq:Stir-recursion} by $[r]_q[2r]_q\cdots [kr]_q$ shows that the right-hand side of equation~\eqref{eq:Hilb-SOP} satisfies exactly the same recursion, and both sides of the second equality in \eqref{eq:Hilb-SOP} equal $\delta_{k,0}$ at $n=0$.
\end{proof}

\begin{corollary}\label{cor:Hilbert-series}
	For $n \geq 1$, we have
	\begin{equation}\label{eq:Hilbert-series}
		\begin{aligned}
			\Hilb\big(R^{(1,1)}_{\Z_r\wr\mathfrak{S}_n};q;z\big)
			 & = \sum_{k=0}^{n} z^{n-k} \, [r]_q [2r]_q \cdots [kr]_q \, \Stir^{(r)}_q(n,k) \\
			 & = \sum_{J \subseteq [n]} z^{|J|} \prod_{i=1}^{n} \, [\stair(J,r)_i + 1]_q.
		\end{aligned}
	\end{equation}
\end{corollary}

\begin{proof}
	The first equality is Proposition~\ref{prop:Hilb-SOP} and equation~\eqref{eq:last-expression}.
	To see the second, define
	\begin{equation}
		P(n,k) := \sum_{\substack{J \subseteq [n],\\ |J| = n-k}} \prod_{i=1}^{n} \, [\stair(J,r)_i + 1]_q
	\end{equation}
	for $ 0 \leq k \leq n$, and $P(n,k) := 0$ otherwise.
	For those $J \subseteq [n]$ of cardinality $n-k$, the first $n-1$ entries of $\stair(J,r)$ form the staircase $\stair(J \cap [n-1],r)$, while the final entry is given by
	\begin{equation}
		\stair(J, r)_n = \begin{cases}
			kr-1   & \text{ if } n \not \in J, \\
			kr+r-2 & \text{ if } n \in J.
		\end{cases}
	\end{equation}
	Thus $P(n,k) = [kr]_q P(n-1,k-1) + [kr+r-1]_q P(n-1,k)$ for $n \geq 1$, which is the same recursion as equation~\eqref{eq:S-recursion}, with the same initial condition $P(0,k) = \delta_{k,0}$.
\end{proof}
The second expression in \eqref{eq:Hilbert-series} is the formula of \cite[Corollary~3.22]{BR}, and the first is
the form in which the Hilbert series was conjectured in \cite[Conjecture~5.3]{SSreflection}.
Their equality is also
shown in \cite[proof of Proposition~5.7]{SSreflection}.

\section*{Acknowledgments}
We would like to thank Andy Wilson and Jennifer Morse for helpful conversations. 
B. Rhoades was partially supported by NSF Grant DMS-2246846.

\section*{Computational tools and AI usage}
Supporting computations were performed in SageMath, Macaulay2, and Python. 
We used AI tools only for writing code in support of certain computations, along with proofreading.

\bibliographystyle{amsplain} 
\bibliography{main}
\end{document}